\documentclass{article}
\usepackage{amsmath,amsthm,amsfonts,amssymb,stmaryrd,authblk}
\usepackage{mathabx,enumerate}
\usepackage{accents}
\usepackage{microtype}
\usepackage{fullpage}
\usepackage{hyperref}

\theoremstyle{plain}
\newtheorem{theorem}{Theorem}[section]
\newtheorem{lemma}[theorem]{Lemma}
\newtheorem{corollary}[theorem]{Corollary}
\newtheorem{proposition}[theorem]{Proposition}
\theoremstyle{definition}
\newtheorem{definition}[theorem]{Definition}

\newtheorem*{theorem:invariance}{Theorem \ref{thm:invariance}}
\newtheorem*{corollary:levy-distance}{Corollary \ref{cor:levy-distance}}
\newtheorem*{corollary:cdf-distance}{Corollary \ref{cor:cdf-distance}}
\newtheorem*{theorem:majority}{Theorem \ref{thm:majority}}
\newtheorem*{theorem:bourgain-slice}{Theorem \ref{thm:bourgain-slice}}
\newtheorem*{theorem:kindler-safra}{Theorem \ref{thm:kindler-safra}}
\newtheorem*{theorem:tEKR-Friedgut}{Theorem \ref{thm:tEKR-Friedgut}}
\newtheorem*{theorem:tEKR-stability}{Theorem \ref{thm:tEKR-stability}}
\newtheorem*{theorem:harmonic-projection}{Theorem \ref{thm:harmonic-projection}}

\providecommand{\RR}{\mathbb{R}}
\providecommand{\FF}{\mathbb{F}}
\providecommand{\ZZ}{\mathbb{Z}}
\providecommand{\cG}{\mathcal{G}}
\providecommand{\cA}{\mathcal{A}}
\providecommand{\cB}{\mathcal{B}}
\providecommand{\cF}{\mathcal{F}}
\providecommand{\cJ}{\mathcal{J}}
\providecommand{\cX}{\mathcal{X}}
\providecommand{\cY}{\mathcal{Y}}
\DeclareMathOperator{\baseInf}{Inf}
\providecommand{\Infc}{\baseInf^c}
\providecommand{\Infs}{\baseInf^s}
\DeclareMathOperator*{\EE}{\mathbb{E}}
\DeclareMathOperator*{\VV}{\mathbb{V}}
\DeclareMathOperator{\basestability}{\mathbb{S}}
\providecommand{\stabilityc}{\basestability^c}
\providecommand{\stabilitys}{\basestability^s}
\DeclareMathOperator{\Cov}{Cov}

\providecommand{\charf}[1]{\mathbf{1}_{#1}}
\DeclareMathOperator{\Sq}{Sq}
\DeclareMathOperator{\dist}{dist}
\DeclareMathOperator{\Nor}{N}
\DeclareMathOperator{\Po}{Po}
\DeclareMathOperator{\Ber}{Ber}
\DeclareMathOperator{\Bin}{Bin}

\DeclareMathOperator{\sgn}{sgn}

\title{Invariance principle on the slice}
\author[1]{Yuval Filmus}
\author[2]{Guy Kindler}
\author[3]{Elchanan Mossel}
\author[4]{Karl Wimmer}
\affil[1]{Technion --- Israel Institute of Technology \\
\texttt{yuvalfi@cs.technion.ac.il}}
\affil[2]{The Hebrew University of Jerusalem, Israel \\
\texttt{gkindler@cs.huji.ac.il}}
\affil[3]{The Wharton School, University of Pennsylvania \\ and University of California, Berkeley \\
\texttt{mossel@wharton.upenn.edu}}
\affil[4]{Duquesne University, Pittsburgh, PA \\
\texttt{wimmerk@duq.edu}}

\begin{document}

\maketitle

\begin{abstract}
 The non-linear invariance principle of Mossel, O'Donnell and Oleszkiewicz establishes that if $f(x_1,\ldots,x_n)$ is a multilinear low-degree polynomial with low influences then the distribution of $f(\cB_1,\ldots,\cB_n)$ is close (in various senses) to the distribution of $f(\cG_1,\ldots,\cG_n)$, where $\cB_i \in_R \{-1,1\}$ are independent Bernoulli random variables and $\cG_i \sim \Nor(0,1)$ are independent standard Gaussians. The invariance principle has seen many application in theoretical computer science, including the \emph{Majority is Stablest} conjecture, which shows that the Goemans--Williamson algorithm for MAX-CUT is optimal under the Unique Games Conjecture.

 More generally, MOO's invariance principle works for any two vectors of hypercontractive random variables $(\cX_1,\ldots,\cX_n),(\cY_1,\ldots,\cY_n)$ such that (i) \emph{Matching moments}: $\cX_i$ and $\cY_i$ have matching first and second moments, (ii) \emph{Independence}: the variables $\cX_1,\ldots,\cX_n$ are independent, as are $\cY_1,\ldots,\cY_n$.

 The independence condition is crucial to the proof of the theorem, yet in some cases we would like to use distributions $(\cX_1,\ldots,\cX_n)$ in which the individual coordinates are not independent. A common example is the uniform distribution on the \emph{slice} $\binom{[n]}{k}$ which consists of all vectors $(x_1,\ldots,x_n) \in \{0,1\}^n$ with Hamming weight~$k$. The slice shows up in theoretical computer science (hardness amplification, direct sum testing), extremal combinatorics (Erd\H{o}s--Ko--Rado theorems) and coding theory (in the guise of the Johnson association scheme).

 Our main result is an invariance principle in which $(\cX_1,\ldots,\cX_n)$ is the uniform distribution on a slice $\binom{[n]}{pn}$ and $(\cY_1,\ldots,\cY_n)$ consists either of $n$ independent $\Ber(p)$ random variables, or of $n$ independent $\Nor(p,p(1-p))$ random variables. As applications, we prove a version of \emph{Majority is Stablest} for functions on the slice, a version of Bourgain's tail theorem, a version of the Kindler--Safra structural theorem, and a stability version of the $t$-intersecting Erd\H{o}s--Ko--Rado theorem, combining techniques of Wilson and Friedgut.

 Our proof relies on a combination of ideas from analysis and probability, algebra and combinatorics. In particular, we make essential use of recent work of the first author which describes an explicit Fourier basis for the slice.
\end{abstract}

\thispagestyle{empty}
\pagebreak
\setcounter{page}{1}

\section{Introduction} \label{sec:intro}

Analysis of Boolean functions is an area at the intersection of theoretical computer science, functional analysis and probability theory, which traditionally studies Boolean functions on the Boolean cube $\{0,1\}^n$. A recent development in the area is the non-linear  \emph{invariance principle} of Mossel, O'Donnell and Oleszkiewicz~\cite{MOO}, a vast generalization of the fundamental Berry--Esseen theorem. The Berry--Esseen theorem is a quantitative version of the Central Limit Theorem, giving bounds on the speed of convergence of a sum $\sum_i X_i$ to the corresponding Gaussian distribution. Convergence occurs as long as none of the summands $X_i$ is too ``prominent''. The invariance principle is an analog of the Berry--Esseen theorem for low-degree polynomials. Given a low-degree polynomial $f$ on $n$ variables in which none of the variables is too prominent (technically, $f$ has low \emph{influences}), the invariance principle states that the distribution of $f(X_1,\ldots,X_n)$ and $f(Y_1,\ldots,Y_n)$ is similar as long as each of the vectors $(X_1,\ldots,X_n)$ and $(Y_1,\ldots,Y_n)$ consists of independent coordinates, the distributions of $X_i,Y_i$ have matching first and second moments, and the variables $X_i,Y_i$ are hypercontractive.

The invariance principle came up in the context of proving a conjecture, \emph{Majority is Stablest}, claiming that the majority function is the most noise stable among functions which have low influences. It is often applied in the following setting: the $X_i$ are skewed Bernoulli variables, and the $Y_i$ are the matching normal distributions. The invariance principle allows us to analyze a function on the Boolean cube (corresponding to the $X_i$) by analyzing its counterpart in Gaussian space (corresponding to the $Y_i$), in which setting it can be analyzed using geometric methods. This approach has been used to prove many results in analysis of Boolean functions (see for example~\cite{KhotSurvey}).

The proof of the invariance principle relies on the product structure of the underlying probability spaces. The challenge of proving an invariance principle for non-product spaces seems far from trivial.
Here we prove such an invariance principle for the
distribution over $X_1,\ldots,X_n$ which is uniform over the \emph{slice} $\binom{[n]}{k}$, defined as:
\[ \binom{[n]}{k} = \{ (x_1,\ldots,x_n) \in \{0,1\}^n : x_1+\cdots+x_n = k \}. \]
This setting arises naturally in hardness of approximation, see e.g.~\cite{DDGKS}, and in extremal combinatorics (the Erd\H{o}s--Ko--Rado theorem and its many extensions).

 Our invariance principle states that if $f$ is a low-degree function on $\binom{[n]}{k}$ having low influences, then the distributions of $f(X_1,\ldots,X_n)$ and $f(Y_1,\ldots,Y_n)$ are close, where $X_1,\ldots,X_n$ is the uniform distribution on $\binom{[n]}{k}$, and $Y_1,\ldots,Y_n$ are either independent Bernoulli variables with expectation $k/n$, or independent Gaussians with the same mean and variance.

The classical invariance principle is stated only for low-influence functions. Indeed, high-influence functions like $f(x_1,\ldots,x_n) = x_1$ behave very differently on the Boolean cube and on Gaussian space. For the same reason, the condition of low-influence is necessary when comparing functions on the slice and on Gaussian space.

The invariance principle allows us to generalize two fundamental results to this setting: Majority is Stablest and Bourgain's tail bound. Using Bourgain's tail bound, we prove an analog of the Kindler--Safra theorem, which states that if a Boolean function is close to a function of constant degree, then it is close to a junta.

As a corollary of our Kindler--Safra theorem, we prove a stability version of the $t$-intersecting
Erd\H{o}s--Ko--Rado theorem, combining the method of Friedgut~\cite{Friedgut} with calculations of Wilson~\cite{Wilson}. Friedgut showed that a $t$-intersecting family in $\binom{[n]}{k}$ of almost maximal size $(1-\epsilon)\binom{n-t}{k-t}$ is close to an optimal family (a $t$-star) as long as $\lambda < k/n < 1/(t+1) - \zeta$ (when $k/n > 1/(t+1)$, $t$-stars are no longer optimal). We extend his result to the regime $k/n \approx 1/(t+1)$.

\smallskip

The classical invariance principle is stated for \emph{multilinear} polynomials, implicitly relying on the fact that every function on $\{0,1\}^n$ can be represented (uniquely) as a multilinear polynomial, and that multilinear polynomials have the same mean and variance under any product distribution in which the individual factors have the same mean and variance. In particular, the classical invariance principle shows that the correct way to lift a low-degree, low-influence function from $\{0,1\}^n$ to Gaussian space is via its multilinear representation.

The analogue of the collection of low degree multilinear functions on the discrete cube is given by the collection of low degree multilinear polynomials annihilated by the operator $\sum_{i=1}^n \frac{\partial}{\partial x_i}$. Dunkl~\cite{Dunkl76,Dunkl79} showed that every function on the slice has a unique representation as a multilinear polynomial annihilated by the operator $\sum_{i=1}^n \frac{\partial}{\partial x_i}$. We call a polynomial satisfying this condition a \emph{harmonic function}. In a recent paper~\cite{F}, the first author showed that low-degree harmonic functions have \emph{similar} mean and variance under both the uniform distribution on the slice and the corresponding Bernoulli and Gaussian product distributions. This is a necessary ingredient in our invariance principle. 

Our results also apply for function on the slice that are not written in their harmonic representation. Starting with an arbitrary multilinear polynomial $f$, there is a unique harmonic function $\tilde f$ agreeing with $f$ on a given slice. We show that as long as $f$ depends on few coordinates, the two functions $f$ and $\tilde f$ are close as functions over the Boolean cube. This implies that $f$ behaves similarly on the slice, on the Boolean cube, and on Gaussian space.

\smallskip

Our proof combines algebraic, geometric and analytic ideas.
A coupling argument, which crucially relies on properties of harmonic functions, shows that the distribution of a low-degree, low-influence harmonic function $f$ is approximately invariant when we move from the original slice to nearby slices. Taken together, these slices form a thin layer around the original slice, on which $f$ has roughly the same distribution as on the original slice.
The classical invariance principle implies that the distribution of $f$ on the layer is close to its distribution on the Gaussian counterpart of the layer, which turns out to be \emph{identical} to its distribution on all of Gaussian space, completing the proof.

\smallskip

A special case of our main result can be stated as follows.

\begin{theorem} \label{cor:cdf-distance_main}
 For every $\epsilon > 0$ and integer $d \geq 0$ there exists $\tau = \tau(\epsilon,d) > 0$ such that the following holds.
 Let $n \geq 1/\tau$, and let $f$ be a harmonic multilinear polynomial of degree $d$ such that with respect to  the uniform measure $\nu_{pn}$ on the slice $\binom{[n]}{pn}$, the variance of $f$ is at most $1$ and all influences of $f$ are bounded by $\tau$.

 The CDF distance between the distribution of $f$ on the slice $\nu_{pn}$ and the distribution of $f$ under the product measure $\mu_p$ with marginals $\Ber(p)$ is at most $\epsilon$: for all $\sigma \in \RR$,
\[
 |\Pr_{\nu_{pn}}[f < \sigma] - \Pr_{\mu_p}[f < \sigma]| < \epsilon.
\]
\end{theorem}

\noindent This result is proved in Section~\ref{sec:invariance-cor}.

\smallskip

Subsequent to this work, the first and third author came up with an alternative proof of Theorem~\ref{cor:cdf-distance_main}~\cite{FilmusMossel} which doesn't require the influences of $f$ to be bounded. The proof is completely different, connecting the measures $\mu_p$ and $\nu_{pn}$ directly without recourse to Gaussian space. While the main result of~\cite{FilmusMossel} subsumes the main result of this paper, we believe that both approaches have merit. Furthermore, the applications of the invariance principle appearing here are not reproduced in~\cite{FilmusMossel}.

\paragraph*{Paper organization} An overview of our main results and methods appears in Section~\ref{sec:overview}. Some preliminaries are described in Section~\ref{sec:prel}. We examine harmonic multilinear polynomials in Section~\ref{sec:harmonic-projection}. We prove the invariance principle in Section~\ref{sec:invariance}. Section~\ref{sec:majority} proves \emph{Majority is Stablest}, and Section~\ref{sec:bourgain} proves Bourgain's tail bound, two applications of the main invariance principle. Section~\ref{sec:kindler-safra} deduces a version of the Kindler--Safra theorem from Bourgain's tail bound. Our stability result for $t$-intersecting families appears in Section~\ref{sec:t-intersecting}.
Some open problems are described in Section~\ref{sec:open-problems}.

\section{Overview} \label{sec:overview}

The goal of this section is to provide an overview of the results proved in this paper and the methods used to prove them. It is organized as follows. Some necessary basic definitions appear in Subsection~\ref{sec:overview:def}. The invariance principle, its proof, and some standard consequences are described in Subsection~\ref{sec:overview:invariance}. Some applications of the invariance principle appear in Subsection~\ref{sec:overview:applications}: versions of Majority is stablest, Bourgain's theorem, and the Kindler--Safra theorem for the slice. An application of the Kindler--Safra theorem to extremal combinatorics is described in Subsection~\ref{sec:overview:intersecting}. Finally, Subsection~\ref{sec:overview:harmonic} presents results for non-harmonic multilinear polynomials.

\subsection{Basic definitions} \label{sec:overview:def}

\paragraph*{Measures}
Our work involves three main probability measures, parametrized by an integer $n$ and a probability $p \in (0,1)$:
\begin{itemize}
\item $\mu_p$ is the product distribution supported on the Boolean cube $\{0,1\}^n$ given by $\mu_p(S) = p^{|S|} (1-p)^{n-|S|}$.
\item $\nu_{pn}$ is the uniform distribution on the slice $\binom{[n]}{pn} = \{ (x_1,\ldots,x_n) \in \{0,1\}^n : x_1 + \cdots + x_n = pn \}$ (we assume $pn$ is an integer).
\item $\cG_p$ is the Gaussian product distribution $\Nor((p,\ldots,p), p(1-p)I_n)$ on Gaussian space $\RR^n$.
\end{itemize}

We denote by $\|f\|_\pi$ the L2 norm of the polynomial $f$ with respect to the measure $\pi$.

\paragraph*{Harmonic polynomials}
As stated in the introduction, we cannot expect an invariance principle to hold for all multilinear polynomials, since for example the polynomial $x_1 + \cdots + x_n - pn$ vanishes on the slice but not on the Boolean cube or on Gaussian space. We therefore restrict our attention to \emph{harmonic} multilinear polynomials, which are multilinear polynomials $f$ satisfying the differential equation
\[
 \sum_{i=1}^n \frac{\partial f}{\partial x_i} = 0.
\]
(The name \emph{harmonic}, whose common meaning is different, was lifted from the literature.)

Dunkl~\cite{Dunkl76,Dunkl79} showed that every function on the slice $\binom{[n]}{pn}$ has a unique representation as a harmonic multilinear polynomial whose degree is at most $\min(pn,(1-p)n)$. This is the analog of the well-known fact that every function on the Boolean cube has a unique representation as a multilinear polynomial.

One crucial property of low-degree harmonic multilinear polynomials is invariance of their L2 norm: for any $p \leq 1/2$ and any harmonic multilinear polynomial $f$ of degree $d \leq pn$,
\[
 \|f\|_{\mu_p} = \|f\|_{\cG_p} = \|f\|_{\nu_{pn}} \left(1 \pm O\left(\frac{d^2}{p(1-p)n}\right)\right).
\]
This is proved in Filmus~\cite{F}, and in fact this result (and its applications in the present work) was the main motivation for~\cite{F}.

\paragraph*{Influences}
The classical definition of influence for a function $f$ on the Boolean cube goes as follows. Define $f^{[i]}(x) = f(x^{[i]})$, where $x^{[i]}$ results from flipping the $i$th coordinate of $x$. The $i$th \emph{cube}-influence of $f$ is given by
\[
 \Infc_i[f] = \|f-f^{[i]}\|_{\mu_p}^2 = \left\|\frac{\partial f}{\partial x_i}\right\|_{\mu_p}^2 = \frac{1}{p(1-p)} \sum_{i \in S} \hat{f}(S)^2.
\]

This notion doesn't make sense for functions on the slice, since the slice is not closed under flipping of a single coordinate. Instead, we consider what happens when two coordinates are swapped. Define $f^{(ij)}(x) = f(x^{(ij)})$, where $x^{(ij)}$ results from swapping the $i$th and $j$th coordinates of $x$. The $(i,j)$th \emph{slice}-influence of $f$ is given by
\[
 \Infs_{ij}[f] = \EE_{\nu_{pn}}[(f - f^{(ij)})^2].
\]
The influence of a single coordinate $i$ is then defined as
\[
 \Infs_i[f] = \frac{1}{n} \sum_{j=1}^n \Infs_{ij}[f].
\]

The two definitions are related: Lemma~\ref{lem:derivative-norm} shows that if $d = O(\sqrt{n})$ then
\[
 \Infs_i[f] = O_p\left(\frac{d}{n} \VV[f] + \Infs_c[f]\right).
\]
(The variance can be taken with respect to either the Boolean cube or the slice, due to the L2 invariance property.)

\paragraph*{Noise stability}
The classical definition of noise stability for a function $f$ on the Boolean cube goes as follows:
\[
 \stabilityc_\rho[f] = \EE[f(x) f(y)],
\]
where $x \sim \mu_p$ and $y$ is obtained from $x$ by letting $y_i = x_i$ with probability $\rho$, and $y_i \sim \mu_p$ otherwise.

The analogous definition on the slice is slightly more complicated. For a function $f$ on the slice,
\[
 \stabilitys_\rho[f] = \EE[f(x) f(y)],
\]
where $x \sim \nu_{pn}$ and $y$ is obtained from $x$ by doing $\Po(\frac{n-1}{2}\log\frac{1}{\rho})$ random transpositions (here $\Po(\lambda)$ is a Poisson distribution with mean $\lambda$). That this definition is the correct analog can be seen through the spectral lens:
\[
 \stabilityc_\rho[f] = \sum_d \rho^d \|f^{=d}\|_{\mu_p}^2, \qquad
 \stabilitys_\rho[f] = \sum_d \rho^{d-d(d-1)/n} \|f^{=d}\|_{\mu_{pn}}^2.
\]

Here $f^{=d}$ is the $d$th homogeneous part of $f$ consisting of all monomials of degree $d$.

\subsection{Invariance principle} \label{sec:overview:invariance}

Our main theorem is an invariance principle for the slice.

\begin{theorem:invariance}
 Let $f$ be a harmonic multilinear polynomial of degree $d$ such that with respect to $\nu_{pn}$, $\VV[f] \leq 1$ and $\Infs_i[f] \leq \tau$ for all $i \in [n]$.
 Suppose that $\tau \leq I_p^{-d}\delta^K$ and $n \geq I_p^d/\delta^K$, for some constants $I_p,K$.
 For any $C$-Lipschitz functional $\psi$ and for $\pi \in \{\cG_p,\mu_p\}$,
\[
 |\EE_{\nu_{pn}}[\psi(f)] - \EE_{\pi}[\psi(f)]| = O_p(C\delta).
\]
\end{theorem:invariance}
\begin{proof}[Proof sketch]
Let $\psi$ be a Lipschitz functional and $f$ a harmonic multilinear polynomial of unit variance, low slice-influences, and low degree $d$.
A simple argument (mentioned above) shows that $f$ also has low cube-influences, and this implies that
\[
 \EE_{\nu_k}[\psi(f)] \approx \EE_{\nu_{pn}}[\psi(f)] \pm O_p\left(\frac{|k-np|}{\sqrt{n}} \cdot \sqrt{d}\right).
\]
The idea is now to apply the multidimensional invariance principle jointly to $f$ and to $S = \frac{x_1+\cdots+x_n-np}{\sqrt{p(1-p)n}}$, deducing
\[
 \EE_{\mu_p} [\psi(f) \charf{|S| \leq \sigma}] = \EE_{\cG_p} [\psi(f) \charf{|S| \leq \sigma}] \pm \epsilon.
\]
Let $\gamma_{p,q}$ be the restriction of $\cG_p$ to the Gaussian slice $\{ (x_1,\ldots,x_n) \in \RR^n : x_1 + \cdots + x_n = qn \}$. An easy argument shows that since $f$ is harmonic, the distribution of $f(\cG_p)$ and $f(\gamma_{p,q})$ is identical, and so
\[
 \EE_{\cG_p} [\psi(f) \charf{|S| \leq \sigma}] = \Pr_{\cG_p}[|S| \leq \sigma] \EE_{\cG_p} [\psi(f)].
\]
Similarly,
\[
 \EE_{\mu_p}[\psi(f) \charf{|S| \leq \sigma}] = \Pr_{\mu_p}[|S| \leq \sigma] (\EE_{\mu_p} [\psi(f)] \pm O_p(\sigma \sqrt{d})).
\]
Since $\Pr_{\cG_p}[|S| \leq \sigma] \approx \Pr_{\mu_p}[|S| \leq \sigma] = \Theta_p(\sigma)$, we can conclude that
\[
 \EE_{\nu_{pn}} [\psi(f)] \approx \EE_{\cG_p} [\psi(f)] \pm O_p\left(\sigma \sqrt{d} + \frac{\epsilon}{\sigma}\right).
\]
By choosing $\sigma$ appropriately, we balance the two errors and obtain our invariance principle.
\end{proof}

As corollaries, we bound the L\'evy and CDF distances between $f(\nu_{pn})$, $f(\mu_p)$ and $f(\cG_p)$:
\begin{corollary:levy-distance}
 Let $f$ be a harmonic multilinear polynomial of degree $d$ such that with respect to $\nu_{pn}$, $\VV[f] \leq 1$ and $\Infs_i[f] \leq \tau$ for all $i \in [n]$.
 There are parameters $X_p,X$ such that for any $0 < \epsilon < 1/2$, if $\tau \leq X_p^{-d}\epsilon^X$ and $n \geq X_p^d/\epsilon^X$ then the L\'evy distance between $f(\nu_{pn})$ and $f(\pi)$ is at most $\epsilon$, for $\pi \in \{\cG_p,\mu_p\}$. In other words, for all $\sigma$,
\[
 \Pr_{\nu_{pn}}[f \leq \sigma - \epsilon] - \epsilon \leq \Pr_{\pi}[f \leq \sigma] \leq \Pr_{\nu_{pn}}[f \leq \sigma + \epsilon] + \epsilon.
\]
\end{corollary:levy-distance}

\begin{corollary:cdf-distance}
 Let $f$ be a harmonic multilinear polynomial of degree $d$ such that with respect to $\nu_{pn}$, $\VV[f] = 1$ and $\Infs_i[f] \leq \tau$ for all $i \in [n]$.
 There are parameters $Y_p,Y$ such that for any $0 < \epsilon < 1/2$, if $\tau \leq (Y_pd)^{-d}\epsilon^{Yd}$ and $n \geq (Y_pd)^d/\epsilon^{Yd}$ then the CDF distance between $f(\nu_{pn})$ and $f(\pi)$ is at most $\epsilon$, for $\pi \in \{\cG_p,\mu_p\}$. In other words, for all $\sigma$,
\[
 |\Pr_{\nu_{pn}}[f \leq \sigma] - \Pr_{\pi}[f \leq \sigma]| \leq \epsilon.
\]
\end{corollary:cdf-distance}

The proofs of these corollaries closely follows the proof of the analogous results in~\cite{MOO}.

\subsection{Applications} \label{sec:overview:applications}

As applications to our invariance principle, we prove analogues of three classical results in analysis of Boolean functions: Majority is stablest; Bourgain's theorem; and the Kindler--Safra theorem:

\begin{theorem:majority} 
 Let $f\colon \binom{[n]}{pn} \to [0,1]$ have expectation $\mu$ and satisfy $\Infs_i[f] \leq \tau$ for all $i \in [n]$. For any $0 < \rho < 1$, we have
\[
 \stabilitys_\rho[f] \leq \Gamma_\rho(\mu) + O_{p,\rho}\left(\frac{\log\log\frac{1}{\alpha}}{\log\frac{1}{\alpha}}\right) + O_\rho\left(\frac{1}{n}\right), \text{ where } \alpha = \min(\tau,\tfrac{1}{n}),
\]
 where $\Gamma_\rho(\mu)$ is the probability that two $\rho$-correlated Gaussians be at most $\Phi^{-1}(\mu)$ (here $\Phi$ is the CDF of a standard Gaussian).
\end{theorem:majority}

\begin{theorem:bourgain-slice} 
 Fix $k \geq 2$. Let $f\colon \binom{[n]}{pn} \to \{\pm 1\}$ satisfy $\Infs_i[f^{\leq k}] \leq \tau$ for all $i \in [n]$.
 For some constants $W_{p,k},C$, if $\tau \leq W_{p,k}^{-1}\VV[f]^C$ and $n \geq W_{p,k}/\VV[f]^C$ then
\[ \|f^{>k}\|^2 = \Omega\left(\frac{\VV[f]}{\sqrt{k}}\right). \]
\end{theorem:bourgain-slice}

\begin{theorem:kindler-safra} 
 Fix the parameter $k \geq 2$. Let $f\colon \binom{[n]}{pn} \to \{\pm 1\}$ satisfy $\|f^{>k}\|^2 = \epsilon$. There exists a function $h\colon \binom{[n]}{pn} \to \{\pm 1\}$ of degree $k$ depending on $O_{k,p}(1)$ coordinates (that is, invariant under permutations of all other coordinates) such that
 \[ \|f-h\|^2 = O_{p,k}\left(\epsilon^{1/C} + \frac{1}{n^{1/C}}\right), \]
 for some constant $C$.
\end{theorem:kindler-safra}

The proof of Theorem~\ref{thm:majority} closely follows its proof in~\cite{MOO}. The proofs of the other two theorems closely follows analogous proofs in~\cite{KKO}.

\subsection{\texorpdfstring{$t$}{t}-Intersecting families} \label{sec:overview:intersecting}

As an application of our Kindler--Safra theorem, we prove a stability result for $t$-intersecting families.

First, a few definitions:

\begin{itemize}
 \item A $t$-intersecting family $\cF \subseteq \binom{[n]}{k}$ is one in which $|A \cap B| \geq t$ for any $A,B \in \cF$.
 \item A $t$-star is a family of the form $\{A \in \binom{[n]}{k} : A \supseteq J\}$, where $|J|=t$.
 \item A $(t,1)$-Frankl family is a family of the form $\{A \in \binom{[n]}{k} : |A \cap J| \geq t+1\}$, where $|J|=t+2$.
\end{itemize}

Ahlswede and Khachatrian~\cite{AK2,AK3} proved that if $n > (t+1)(k-t+1)$ and $\cF$ is an intersecting family, then $|\cF| \leq \binom{n-t}{k-t}$, and furthermore equality holds if and only if $\cF$ is a $t$-star. They also proved that when $n = (t+1)(k-t+1)$ the same upper bound holds, but now equality holds for both $t$-stars and $(t,1)$-Frankl families.

A corresponding stability result was proved by Friedgut~\cite{Friedgut}:

\begin{theorem:tEKR-Friedgut}
 Let $t \geq 1$, $k \geq t$, $\lambda,\zeta > 0$, and $\lambda n < k < (\frac{1}{t+1} - \zeta) n$.
 Suppose $\cF \subseteq \binom{[n]}{k}$ is a $t$-intersecting family of measure $|\cF| = \binom{n-t}{k-t} - \epsilon \binom{n}{k}$. Then there exists a family $\cG$ which is a $t$-star such that
 \[ \frac{|\cF \triangle \cG|}{\binom{n}{k}} = O_{t,\lambda,\zeta}(\epsilon). \]
\end{theorem:tEKR-Friedgut}

Friedgut's theorem requires $k/n$ to be bounded away from $1/(t+1)$. Using the Kindler--Safra theorem on the slice rather than the Kindler--Safra theorem on the Boolean cube (which is what Friedgut uses), we can do away with this limitation:

\begin{theorem:tEKR-stability}
Let $t \geq 2$, $k \geq t+1$ and $n = (t+1)(k-t+1) + r$, where $r > 0$. Suppose that $k/n \geq \lambda$ for some $\lambda > 0$.
Suppose $\cF \subseteq \binom{[n]}{k}$ is a $t$-intersecting family of measure $|\cF| = \binom{n-t}{k-t} - \epsilon \binom{n}{k}$. Then there exists a family $\cG$ which is a $t$-star or a $(t,1)$-Frankl family such that
\[ \frac{|\cF \triangle \cG|}{\binom{n}{k}} = O_{t,\lambda}\left(\max\left(\left(\frac{k}{r}\right)^{1/C},1\right)\epsilon^{1/C} + \frac{1}{n^{1/C}} \right), \]
for some constant $C$.

Furthermore, there is a constant $A_{t,\lambda}$ such that $\epsilon \leq A_{t,\lambda} \min(r/k,1)^{C+1}$ implies that $\cG$ is a $t$-star.
\end{theorem:tEKR-stability}

Our proof closely follows the argument of Friedgut~\cite{Friedgut}, transplanting it from the setting of the Boolean cube to the setting of the slice, using calculations of Wilson~\cite{Wilson} in the latter setting. The argument involves certain subtelties peculiar to the slice.

\subsection{Non-harmonic functions} \label{sec:overview:harmonic}

All results we have described so far apply only to harmonic multilinear polynomials. We mentioned that some of these results trivially don't hold for some non-harmonic multilinear polynomials: for example, $\sum_{i=1}^n x_i - np$ doesn't exhibit invariance. This counterexample, however, is a function depending on all coordinates. In contrast, we can show that some sort of invariance does apply for general multilinear polynomials that depend on a small number of coordinates:

\begin{theorem:harmonic-projection}
 Let $f$ be a multilinear polynomial depending on $d$ variables, and let $\tilde{f}$ be the unique harmonic multilinear polynomial agreeing with $f$ on $\binom{[n]}{pn}$, where $d \leq pn \leq n/2$. For $\pi \in \{\mu_p,\cG_p\}$ we have
\[
 \|f - \tilde{f}\|_\pi^2 = O\left(\frac{d^22^d}{p(1-p)n}\right)\|f\|_\pi^2.
\]
\end{theorem:harmonic-projection}
\begin{proof}[Proof sketch]
 Direct calculation (appearing in Lemma~\ref{lem:character-invariance}) shows that if $\omega$ is a Fourier character than
\[
 \| \omega - \tilde{\omega} \|_{\mu_p}^2 = \| \omega - \tilde{\omega} \|_{\cG_p}^2 = O\left(\frac{d^2}{p(1-p)n}\right),
\]
 where $\tilde{\omega}$ is defined analogously to $\tilde{f}$.
 
 We can assume without loss of generality that $f$ depends only on the variables in $[d] = \{1,\ldots,d\}$.
 Since $\tilde{f} = \sum_{S \subseteq [d]} \hat{f}(S) \tilde{\omega}_S$,
\[
 \|f - \tilde{f}\|_\pi^2 \leq 2^d \sum_{S \subseteq [d]} \hat{f}(S)^2 O\left(\frac{d^2}{p(1-p)n}\right) = O\left(\frac{d^22^d}{p(1-p)n}\right)\|f\|_\pi^2,
\]
 using the Cauchy--Schwartz inequality.
\end{proof}

The idea of the proof is to prove a similar results for Fourier characters (Lemma~\ref{lem:character-invariance}) for individual Fourier characters, and then to invoke the Cauchy--Schwartz inequality.

As a consequence, if we have a multilinear polynomial $f$ depending on a small number of variables, its harmonic projection $\tilde{f}$ (defined as in the theorem) has a similar expectation, L2 norm, variance and noise stability (Corollary~\ref{cor:harmonic-projection}). This implies, for example, that our Majority is stablest theorem is tight: the harmonic projection of the majority of a small number of indices serves as the tight example.

\section{Preliminaries} \label{sec:prel}

\paragraph*{Notation} The notation $\charf{E}$ is the characteristic function of the event $E$. Expectation, variance and covariance are denoted by $\EE$, $\VV$ and $\Cov$, respectively.
The sign function is denoted $\sgn$.
The notation $[n]$ denotes the set $\{1,\ldots,n\}$. The \emph{slice} $\binom{[n]}{k}$ consists of all subsets of $[n]$ of cardinality $k$. We often identify subsets of $[n]$ with their characteristic vectors in $\{0,1\}^n$.

The notation $\Bin(n,p)$ denotes a binomial distribution with $n$ trials and success probability $p$. The notation $\Po(\lambda)$ denotes a Poisson distribution with expectation $\lambda$. The notation $\Nor(\mu,\Sigma^2)$ denotes a normal distribution with mean $\mu$ and covariance matrix $\Sigma^2$. For a scalar $p$, we use $\mathbf{p}$ to denote a constant $p$ vector (of appropriate dimension which is clear from context) and $I_n$ to denote the $n\times n$ identity matrix.

For a probability distribution $\pi$, $ \| f \| = \|f\|_\pi = \sqrt{\EE_\pi[f^2]}$ is the L2~norm of $f$ with respect to $\pi$.  Note that $\|f\|_1 = \EE[|f|]$.

The notation $a^{\underline{b}}$ denotes the falling factorial function: $a^{\underline{b}} = a(a-1)\cdots(a-b+1)$.

Asymptotic notation ($O(\cdot)$ and the like) will always denote non-negative expressions. When the expression can be positive or negative, we use the notation $\pm O(\cdot)$. The underlying limit is always $n\to\infty$. If the hidden constant depends on variables $V$, we use the notation $O_V(\cdot)$.

A \emph{$C$-Lipschitz functional} is a function $\psi\colon\RR\to\RR$ satisfying $|\psi(x)-\psi(y)| \leq C|x-y|$, which implies that for functions
$f,g$ on the same domain:
\begin{lemma} \label{lem:lipschitz-cs}
For every $C$-Lipschitz functional $\psi$ and functions $f,g$ on the same domain,
\[ |\EE[\psi(f)] - \EE[\psi(g)]| \leq C\|f-g\|. \]
\end{lemma}

\paragraph*{Probability distributions} Our argument will involve several different probability distributions on $\RR^n$ (where $n$ will always be clear from context):
\begin{itemize}
\item $\mu_p$ is the product distribution supported on $\{0,1\}^n$ given by $\mu_p(S) = p^{|S|} (1-p)^{n-|S|}$.
\item $\nu_k$ is the uniform distribution on the slice $\binom{[n]}{k}$.
\item $\cG_p$ is the Gaussian product distribution $\Nor((p,\ldots,p), p(1-p)I_n)$.
\item $\gamma_{p,q} = \Nor((q,\ldots,q), \Sigma)$, where $\Sigma_{i,j} = \frac{n-1}{n} p(1-p) \delta(i = j)  -\frac{n-1}{n} \frac{p(1-p)}{n-1} \delta(i \neq j)$ for $1 \leq i,j \leq n$.
\end{itemize}

As is well-known, the distribution $\gamma_{p,q}$ results from conditioning $\cG_p$ on the sum being $qn$.

\begin{lemma} \label{lem:gamma-conditioning}
Let $(X_1,\ldots,X_n) \sim \cG_p$. The distribution of $(X_1,\ldots,X_n)$ conditioned on $X_1+\cdots+X_n=qn$ is $\gamma_{p,q}$.
\end{lemma}
\begin{proof}
Let $S=X_1+\cdots+X_n$, and consider the multivariate Gaussian distribution $(X_1,\ldots,X_n,S)$, whose distribution is easily calculated to be $\Nor(\begin{pmatrix} \mathbf{p} & pn \end{pmatrix}, p(1-p) \begin{pmatrix} I_n & \mathbf{1} \\ \mathbf{1}' & n \end{pmatrix})$. Let $(Y_1,\ldots,Y_n)$ be the distribution of $(X_1,\ldots,X_n,S)$ conditioned on $S = qn$, which is well-known to be multivariate Gaussian. Using well-known formulas, the mean of this distribution is $\mathbf{p} + \mathbf{1} n^{-1} (qn-pn) = \mathbf{q}$ (as can be derived directly), and its covariance matrix is $p(1-p)(I_n - \mathbf{1} n^{-1} \mathbf{1}')$. The diagonal elements are $\VV[Y_i] = p(1-p)(1 - \frac{1}{n})$ and the off-diagonal ones are $\Cov(Y_i,Y_j) = p(1-p)(-\frac{1}{n})$.
\end{proof}

We can also go in the other direction.

\begin{lemma} \label{lem:gamma-unconditioning}
Let $(X_1,\ldots,X_n) \sim \gamma_{p,q}$, let $Y \sim \Nor(p-q,\frac{p(1-p)}{n})$, and let $Y_i = X_i + Y$. Then $(Y_1,\ldots,Y_n) \sim \cG_p$.
\end{lemma}
\begin{proof}
As is well-known, $Y_1,\ldots,Y_n$ is a multivariate Gaussian, and it is easy to see that its mean is $\mathbf{p}$. We have $\VV[Y_i] = \VV[X_i] + \VV[Y] = p(1-p)$ and $\Cov(Y_i,Y_j) = \Cov(X_i,X_j) + \VV[Y] = 0$. The lemma follows.
\end{proof}

The distributions $\mu_p$ and $\nu_k$ are very close for events depending on $o(\sqrt{n})$ coordinates.

\begin{lemma} \label{lem:measures-close}
 Let $A$ be an event depending on $J$ coordinates, where $J^2 \leq n$.
 Then \[ |\nu_{pn}(A) - \mu_p(A)| \leq \frac{J^2}{4p(1-p)n} \mu_p(A). \]
\end{lemma}
\begin{proof}
 The triangle inequality shows that we can assume that $A$ is the event $x_1 = \cdots = x_\ell = 0$, $x_{\ell+1},\ldots,x_J = 1$ for some $\ell$.
 Let $k = pn$. Clearly $\mu_p(A) = (1-p)^\ell p^{J-\ell}$, whereas
\[
 \nu_{pn}(A) = \frac{(n-k)^{\underline{\ell}} k^{\underline{J-\ell}}}{n^{\underline{J}}}.
\]
 We have
\[
 \frac{n^{\underline{J}}}{n^J} = \left(1-\frac{1}{n}\right)\cdots\left(1-\frac{J-1}{n}\right) \geq 1 - \frac{1+\cdots+(J-1)}{n} \geq 1 - \frac{J^2}{2n}.
\]
 Therefore
\[
 \nu_{pn}(A) \leq \frac{(n-k)^\ell k^{J-\ell}}{n^J (1 - J^2/(2n))} \leq \mu_p(A) \left(1 + \frac{J^2}{n}\right),
\]
 using $\frac{1}{1-x} \leq 1+2x$, which is valid for $x \leq 1/2$.

 Similarly,
\[
 \frac{(n-k)^{\underline{\ell}} k^{\underline{J-\ell}}}{(n-k)^\ell k^\ell} \geq 1 - \frac{\ell^2}{2(1-p)n} - \frac{(J-\ell)^2}{2pn} \geq
 1 - \max\left(\frac{J^2}{2(1-p)n}, \frac{J^2}{2pn}\right) \geq 1 - \frac{J^2}{4p(1-p)n}.
\]
 Therefore
\[
 \nu_{pn}(A) \geq \frac{(n-k)^\ell k^\ell (1-J^2/(4p(1-p)n))}{n^\ell} = \mu_p(A) \left(1 - \frac{J^2}{4p(1-p)n}\right).
\]
 This completes the proof.
\end{proof}

\subsection{Harmonic multilinear polynomials} \label{sec:harmonic}

Our argument involves extending a function over a slice $\binom{[n]}{k}$ to a function on $\RR^n$, just as in the classical invariance principle, a function on $\{0,1\}^n$ is extended to $\RR^n$ by writing it as a multilinear polynomial. In our case, the correct way of extending a function over a slice to $\RR^n$ is by interpreting it as a \emph{harmonic multilinear polynomial}. Our presentation follows~\cite{F}, where the proofs of various results claimed in this section can be found. The basis in Definition~\ref{def:harmonic-expansion} below also appears in earlier work of Srinivasan~\cite{Srinivasan}, who constructed it and showed that it is orthogonal with respect to all exchangeable measures.

\begin{definition} \label{def:harmonic}
Let $f \in \RR[x_1,\ldots,x_n]$ be a formal polynomial. We say that $f$ is \emph{multilinear} if $\frac{\partial^2 f}{\partial x_i^2} = 0$ for all $i \in [n]$. We say that $f$ is \emph{harmonic} if
\[ \sum_{i=1}^n \frac{\partial f}{\partial x_i} = 0. \]
\end{definition}

The somewhat mysterious condition of harmonicity arises naturally from the representation theory of the Johnson association scheme. Just as any function on the Boolean cube $\{0,1\}^n$ can be represented uniquely as a multilinear polynomial (up to an affine transformation, this is just the Fourier--Walsh expansion), every function on the slice $\binom{[n]}{k}$ can be represented uniquely as a harmonic multilinear polynomial, using the identification
\[ \binom{[n]}{k} = \{(x_1,\ldots,x_n) \in \{0,1\}^n : \sum_{i=1}^n x_i = k\}. \]

\begin{lemma}[{\cite[Theorem 4.1]{F}}] \label{lem:slice-harmonic}
Every real-valued function $f$ on the slice $\binom{[n]}{k}$ can be represented uniquely as a harmonic multilinear polynomial of degree at most $\min(k,n-k)$.
\end{lemma}

There is a non-canonical Fourier expansion defined for harmonic multilinear polynomials.

\begin{definition} \label{def:harmonic-expansion}
Let $A=(a_1,\ldots,a_d)$ and $B=(b_1,\ldots,b_d)$ be two sequences of some common length $d$ of distinct elements of $[n]$. We say that $A < B$ if:
\begin{enumerate}[(a)]
\item $A$ and $B$ are disjoint.
\item $B$ is monotone increasing: $b_1 < \cdots < b_d$.
\item $a_i < b_i$ for all $i \in [d]$.
\end{enumerate}
A sequence $B=(b_1,\ldots,b_d)$ is a \emph{top set} if $A < B$ for some sequence $A$. The collection of all top sets of length $d$ is denoted $\cB_{n,d}$, and the collection of all top sets is denoted $\cB_n$.

If $A=(a_1,\ldots,a_d)$ and $B=(b_1,\ldots,b_d)$ satisfy $A < B$, define
\[ \chi_{A,B} = \prod_{i=1}^d (x_{a_i} - x_{b_i}). \]
For a top set $B$, define
\[ \chi_B = \sum_{A < B} \chi_{A,B}. \]
Finally, define
\[ \chi_d = \chi_{\{2,4,\ldots,2d\}}. \]
\end{definition}

\begin{lemma}[{\cite[Theorem 3.1,Theorem 3.2]{F}}] \label{lem:harmonic-expansion}
Let $\pi$ be any exchangeable distribution on $\RR^n$ (that is, $\pi$ is invariant under permutation of the coordinates). The collection $\cB_n$ forms an orthogonal basis for all harmonic multilinear polynomials in $\RR[x_1,\ldots,x_n]$ (with respect to $\pi$), and
\[ \|\chi_B\|_\pi^2 = c_B \|\chi_{|B|}\|_\pi^2, \text{where } c_B = \prod_{i=1}^n \binom{b_i-2(i-1)}{2}, \]
and $\|\chi_B\|_\pi$ denotes the norm of $\chi_B$ with respect to $\pi$.

In particular, if $f$ is a harmonic multilinear polynomial then $\EE[f]$ is the same under all exchangeable measures.
\end{lemma}

Lemma~\ref{lem:gamma-unconditioning} and Lemma~\ref{lem:harmonic-expansion} put together have the surprising consequence that harmonic multilinear functions have exactly the same distribution under $\cG_p$ and $\gamma_{p,q}$.

\begin{lemma} \label{lem:harmonic-gauss}
 Let $f$ be a harmonic multilinear polynomial. The random variables $f(\cG_p),\allowbreak f(\gamma_{p,q})$ are identically distributed.
\end{lemma}
\begin{proof}
 According to Lemma~\ref{lem:gamma-unconditioning}, if $(x_1,\ldots,x_n) \sim \gamma_{p,q}$, $y \sim \Nor(q-p,\frac{p(1-p)}{n})$ and $y_i = x_i+y$, then $(y_1,\cdots,y_n) \sim \cG_p$. The lemma follows since $y_i-y_j = x_i-x_j$ and a harmonic multilinear polynomial can be expressed as a function of the differences $x_i-x_j$ for all $i,j$.
\end{proof}

Lemma~\ref{lem:harmonic-expansion} allows us to compare the norms of a harmonic multilinear function under various distributions.

\begin{corollary} \label{cor:norms}
Let $\pi_1,\pi_2$ be two exchangeable distributions on $\RR^n$, and let $f$ be a harmonic multilinear polynomial of degree $d$. If for some $\epsilon \geq 0$ and all $0 \leq e \leq d$ it holds that $(1-\epsilon) \|\chi_e\|_{\pi_1}^2 \leq \|\chi_e\|_{\pi_2}^2 \leq (1+\epsilon) \|\chi_e\|_{\pi_1}^2$, then also $(1-\epsilon) \|f\|_{\pi_1}^2 \leq \|f\|_{\pi_2}^2 \leq (1+\epsilon) \|f\|_{\pi_1}^2$.
\end{corollary}

The following lemma records the norms of basis elements for the distributions considered in this paper.

\begin{lemma} \label{lem:harmonic-basis-norms}
For all $d$ we have
\begin{align*}
\|\chi_d\|_{\mu_p}^2 &= \|\chi_d\|_{\cG_p}^2 = (2p(1-p))^d, \\ 
\|\chi_d\|_{\nu_{pn}}^2 &= 2^d \frac{(pn)^{\underline{d}}((1-p)n)^{\underline{d}}}{n^{\underline{2d}}} = (2p(1-p))^d \left(1 \pm O\left(\frac{d^2}{p(1-p)n}\right) \right).
\end{align*}
\end{lemma}
\begin{proof}
The exact formulas for $\|\chi_d\|_{\mu_p}^2$ and $\|\chi_d\|_{\nu_{pn}}^2$ are taken from~\cite[Theorem 4.1]{F}. Since $x_1,\ldots,x_n$ are independent under $\cG_p$, we have $\|\chi_d\|_{\cG_p}^2 = \EE[(x_1-x_2)^2]^d = (2p(1-p))^d$.

It remains to prove the estimate for $\|\chi_d\|_{\nu_{pn}}^2$. The proof of~\cite[Theorem 4.1]{F} shows that
\[
\|\chi_d\|_{\nu_{pn}}^2 = 2^d \frac{(pn)^{\underline{d}}((1-p)n)^{\underline{d}}}{n^{\underline{2d}}} = (2p(1-p))^d \frac{\left(1-\frac{O(d^2)}{pn}\right) \left(1-\frac{O(d^2)}{(1-p)n}\right)} {\left(1-\frac{O(d^2)}{n}\right)}.
\]
It follows that
\[
\frac{\|\chi_d\|_{\nu_{pn}}^2}{(2p(1-p))^d} = 1 \pm O\left(\frac{d^2}{pn} + \frac{d^2}{(1-p)n} + \frac{d^2}{n}\right) = 1 \pm O\left(\frac{d^2}{p(1-p)n}\right). \qedhere
\]
\end{proof}

Lemma~\ref{lem:harmonic-basis-norms} and Corollary~\ref{cor:norms} imply an L2 invariance principle for low degree harmonic multilinear polynomials.

\begin{corollary} \label{cor:l2-invariance-harmonic}
 Suppose $f$ is a harmonic multilinear polynomial of degree $d$ on $n$ variables. For any $p \leq 1/2$ such that $d \leq pn$ and any $\pi \in \{\mu_p,\cG_p\}$ we have
\[
 \|f\|_{\nu_{pn}} = \|f\|_\pi \left(1 \pm O\left(\frac{d^2}{p(1-p)n}\right)\right).
\]
\end{corollary}

\subsection{Analysis of functions} \label{sec:analysis}

We consider functions on three different kinds of domains: the Boolean cube $\{0,1\}^n$, the slice $\binom{[n]}{k}$, and Gaussian space $\RR^n$. We can view a multilinear polynomial in $\RR[x_1,\ldots,x_n]$ as a function over each of these domains in the natural way.

For each of these domains, we proceed to define certain notions and state some basic results. The material for the Boolean cube and Gaussian space is standard, and can be found for example in~\cite{Ryan}.

\paragraph*{Functions on the Boolean cube} The Boolean cube is analyzed using the measure $\mu_p$ for an appropriate $p$. The Fourier characters $\omega_S$ and Fourier expansion of a function $f\colon\{0,1\}^n\to\RR$ are given by
\[ \omega_S(x_1,\ldots,x_n) = \prod_{i\in S} \frac{x_i-p}{\sqrt{p(1-p)}}, \quad f = \sum_{S\subseteq[n]} \hat{f}(S) \omega_S. \]
We define $f^{=k} = \sum_{|S|=k} \hat{f}(S) \omega_S$, and so a multilinear polynomial $f$ of degree $d$ can be decomposed as $f = f^{=0} + \cdots + f^{=d}$. Since the Fourier characters are orthogonal, the parts $f^{=0},\ldots,f^{=d}$ are orthogonal.
In the future it will be convenient to separate $f$ into $f = f^{\leq k} + f^{>k}$ for an appropriate $k$, where $f^{\leq k} = f^{=0} + \cdots + f^{=k}$ and $f^{>k} = f^{=k+1} + \cdots + f^{=d}$.

Define $f^{[i]}(x) = f(x^{[i]})$, where $x^{[i]}$ results from flipping the $i$th coordinate of $x$.
The $i$th cube-influence is given by
\[
 \Infc_i[f] = \|f-f^{[i]}\|^2 = \left\|\frac{\partial f}{\partial x_i}\right\|^2 = \frac{1}{p(1-p)} \sum_{S\ni i} \hat{f}(S)^2.
\]
 The total influence of $f$ is $\Infc[f] = \sum_{i=1}^n \Infc_i[f]$, and it satisfies the Poincar\'e inequality
\[ \VV[f] \leq p(1-p) \Infc[f] \leq (\deg f) \VV[f]. \]
The noise operator $T_\rho$ is defined by
\[ T_\rho f = \sum_{i=0}^{\deg f} \rho^i f^{=i}. \]
The noise stability of $f$ at $\rho$ is
\[ \stabilityc_\rho[f] = \langle f,T_\rho f \rangle = \sum_{i=0}^{\deg f} \rho^i \|f^{=i}\|^2. \]

The noise operator (and so noise stability) can also be defined non-spectrally. We have $(T_\rho f)(x) = \EE[f(y)]$, where $y$ is obtained from $x$ by letting $y_i = x_i$ with probability $\rho$, and $y_i \sim \mu_p$ otherwise.

\paragraph*{Functions on the slice} The slice $\binom{[n]}{k}$ is analyzed using the measure $\nu_k$. The corresponding notion of Fourier expansion was described in Section~\ref{sec:harmonic}. A harmonic multilinear polynomial $f$ of degree $d$ can be decomposed as $f = f^{=0} + \cdots + f^{=d}$, where $f^{=k}$ contains the homogeneous degree $k$ part. The parts $f^{=0},\ldots,f^{=d}$ are orthogonal.

The $(i,j)$th influence of a function $f$ is $\Infs_{ij}[f] = \EE[(f-f^{(ij)})^2]$, where $f^{(ij)}(x) = f(x^{(ij)})$, and $x^{(ij)}$ is obtained from $x$ by swapping the $i$th and $j$th coordinates. We define the $i$th influence by $\Infs_i[f] = \frac{1}{n} \sum_{j=1}^n \Infs_{ij}[f]$, and the total influence by $\Infs[f] = \sum_{i=1}^n \Infs_i[f]$. The total influence satisfies the Poincar\'e inequality
\[ \VV[f] \leq \Infs[f] \leq (\deg f) \VV[f]. \]
For a proof, see for example~\cite[Lemma 5.6]{F}.

The noise operator $H_\rho$ is defined by
\[ H_\rho f = \sum_{d=0}^{\deg f} \rho^{d(1-(d-1)/n)} f^{=d}. \]
The noise stability of $f$ at $\rho$ is
\[ \stabilitys_\rho[f] = \langle f,H_\rho f \rangle = \sum_{d=0}^{\deg f} \rho^{d(1-(d-1)/n)} \|f^{=d}\|^2. \]

The noise operator (and so noise stability) can also be defined non-spectrally. We have $(H_\rho f)(x) = \EE[f(y)]$, where $y$ is obtained from $x$ by taking $\Po(\frac{n-1}{2}\log\frac{1}{\rho})$ random transpositions.

\paragraph*{Functions on Gaussian space} Gaussian space is $\RR^n$ under a measure $\cG_p$ for an appropriate $p$. In this paper, we mostly consider functions on $\RR^n$ given by multilinear polynomials, and these can be expanded in terms of the $\omega_S$. General functions can be expanded in terms of Hermite functions. Every square-integrable function can be written as $f = \sum_{k \geq 0} f^{=k}$, where $f^{=k}$ satisfies $f^{=k}(\alpha x_1+p,\ldots,\alpha x_n+p) = \alpha^k f^{=k}(x_1+p,\ldots,x_n+p)$.

The distributions $\mu_p$ and $\cG_p$ have the same first two moments, and this implies that $\EE_{\mu_p}[f] = \EE_{\cG_p}[f]$ and $\|f\|_{\mu_p} = \|f\|_{\cG_p}$ for every multilinear polynomial $f$. The Ornstein--Uhlenbeck operator $U_\rho$ is defined just like $T_\rho$ is defined for the cube. Noise stability is defined just like in the case of the cube, and we use the same notation $\stabilityc$ for it.

The noise operator (and so noise stability) can also be defined non-spectrally. We have $(U_\rho f)(x) = \EE[f(y)]$, where $y = (1-\rho) p + \rho x + \sqrt{1-\rho^2}\Nor(0,p(1-p))$. We can also define noise stability as $\stabilityc_\rho[f] = \EE[f(x) f(y)]$, where $(x,y) \sim \Nor\bigl((p,p),\begin{pmatrix} p(1-p) & \rho p(1-p) \\ \rho p(1-p) & p(1-p) \end{pmatrix}\bigr)$.

\paragraph*{Homogeneous parts} For a function $f$, we have defined $f^{=k}$ in three different ways, depending on the domain. When $f$ is a harmonic multilinear polynomial, all three definitions coincide. Indeed, any harmonic multilinear polynomial is a linear combination of functions of the form $\chi_{A,B}$. We show that $\chi_{A,B} = \chi_{A,B}^{=|B|}$ under all three definitions. Let $A = a_1,\ldots,a_k$ and $B = b_1,\ldots,b_k$. Since $\chi_{A,B}$ is homogeneous of degree $k$ as a polynomial, we see that $\chi_{A,B} = \chi_{A,B}^{=k}$ over the slice. Also,
\[
 \chi_{A,B} = (p(1-p))^{k/2} \prod_{i=1}^k \left(\frac{x_{a_i} - p}{\sqrt{p(1-p)}} - \frac{x_{b_i} - p}{\sqrt{p(1-p)}}\right).
\]
Opening the product into a sum of terms, we can identify each term with a basis function $\omega_S$ for some $S$ of size $k$. This shows that $\chi_{A,B} = \chi_{A,B}^{=k}$ over the cube. Finally, since $\chi_{A,B}$ is harmonic, in order to show that $\chi_{A,B} = \chi_{A,B}^{=k}$ in Gaussian space, it suffices to show that $\chi_{A,B}(\alpha x) = \alpha^k \chi_{A,B}(x)$, which is true since $\chi_{A,B}$ is homogeneous of degree $k$ as a polynomial.

\paragraph*{Degrees} The following results state several ways in which degree for functions on the slice behaves as expected.

First, we show that degree is subadditive.

\begin{lemma} \label{lem:slice-degree}
 Let $f,g$ be harmonic multilinear polynomials, and let $h$ be the unique harmonic multilinear polynomial agreeing with $fg$ on the slice $\binom{[n]}{k}$. Then $\deg h \leq \deg f + \deg g$.
\end{lemma}
\begin{proof}
 We can assume that $\deg f + \deg g \leq k$, since otherwise the result is trivial.

 Let $E_i$ be the operator mapping a function $\phi$ on the slice to the function $\phi^{=i}$ on the slice. That is, we take the harmonic multilinear representation of $\phi$, extract the $i$'th homogeneous part, and interpret the result as a function on the slice. Also, let $E_{\leq d} = \sum_{i=0}^d E_i$. A function $\phi$ on the slice has degree at most $d$ if and only if it is in the range of $E_{\leq d}$.

 Qiu and Zhan~\cite{QiuZhan} (see also Tanaka~\cite{Tanaka}) show that $fg$ is in the range of $E_{\leq \deg f} \circ E_{\leq \deg g}$, where $\circ$ is the Hadamard product. The operators $E_i$ are the primitive idempotents of the Johnson association scheme (see, for example,~\cite[\S3.2]{BannaiIto}). Since the Johnson association scheme is Q-polynomial (cometric), the range of $E_{\leq \deg f} \circ E_{\leq \deg g}$ equals the range of $E_{\leq \deg f + \deg g}$, and so $\deg fg \leq \deg f + \deg g$.
\end{proof}

As a corollary, we show that ``harmonic projection'' doesn't increase the degree.

\begin{corollary} \label{cor:slice-degree}
 Let $f$ be a multilinear polynomial, and let $g$ be the unique harmonic multilinear polynomial agreeing with $f$ on the slice $\binom{[n]}{k}$. Then $\deg g \leq \deg f$.
\end{corollary}
\begin{proof}
 When $f = x_1$, one checks that $g$ is given by the linear polynomial
\[
 g = \frac{1}{n} \sum_{i=1}^n (x_1 - x_i) + \frac{k}{n}.
\]
 The corollary now follows from Lemma~\ref{lem:slice-degree} and from the easy observation $\deg (\alpha F + \beta G) \leq \max(\deg F, \deg G)$.
\end{proof}

An immediate corollary is that degree is substitution-monotone.

\begin{corollary} \label{cor:slice-subst}
 Let $f$ be a harmonic multilinear polynomial, let $g(x_1,\ldots,x_n) = \linebreak f(x_1,\ldots,x_{n-1},b)$ for $b \in \{0,1\}$, and let $h$ be the unique harmonic multilinear polynomial agreeing with $g$ on the slice $\binom{[n]}{k}$. Then $\deg h \leq \deg f$.
\end{corollary}

\paragraph*{Noise operators} We have considered two noise operators, $H_\rho$ and $T_\rho = U_\rho$. Both can be applied syntactically on all multilinear polynomials. The following result shows that both operators behave the same from the point of view of Lipschitz functions.

\begin{lemma} \label{lem:noise-operators}
 Let $f$ be a multilinear polynomial of degree at most $n/2$. For $\delta < 1/2$ and any $C$-Lipschitz functional $\psi$, and with respect to any exchangeable measure,
\[
 |\EE[\psi(H_{1-\delta}f)] - \EE[\psi(U_{1-\delta}f)]| = O\left(\frac{C\delta^{-2}}{n}\|f\|\right).
\]
\end{lemma}
\begin{proof}
 Let $\rho = 1-\delta$. Lemma~\ref{lem:lipschitz-cs} shows that
\[
 |\EE[\psi(H_\rho f)] - \EE[\psi(U_\rho f]|^2 \leq C^2 \|H_\rho f - U_\rho f\|^2 = C^2 \sum_{d=0}^{n/2}(\rho^{d(1-(d-1)/n)}-\rho^d)^2 \|f^{=d}\|^2.
\]
 Let $R(x) = \rho^x$. Then $\rho^{d(1-(d-1)/n)} - \rho^d = \frac{d(d-1)}{n}(-R'(x))$ for some $x \in [d(1-(d-1)/n),d]$. For such $x$, $R'(x) = \rho^x (-\log \rho) \leq \rho^{d(1-(d-1)/n)} (2\delta) \leq \rho^{d/2} (2\delta)$, using $\delta < 1/2$ and $d \leq n/2$. Therefore
\[
 \rho^{d(1-(d-1)/n)} - \rho^d \leq 2\delta \frac{d(d-1)}{n} \rho^{d/2}.
\]
 The expansion $x^2/(1-x)^3 = \sum_{d=0}^\infty \binom{d}{2} x^d$ implies that $d(d-1)\rho^{d/2} \leq 2\rho/(1-\sqrt{\rho})^3$. Since $1-\sqrt{\rho} = 1-\sqrt{1-\delta} \geq \delta/2$, we conclude that
\[
 \rho^{d(1-(d-1)/n)} - \rho^d \leq \frac{32\delta^{-2}}{n}.
\]
 The lemma follows.
\end{proof}

\section{On harmonicity} \label{sec:harmonic-projection}

Let $f$ be a function on the Boolean cube $\{0,1\}^n$, and let $\tilde f$ be the unique harmonic function agreeing with $f$ on the slice $\binom{[n]}{pn}$. We call $\tilde f$ the \emph{harmonic projection} of $f$ with respect to the slice $\binom{[n]}{pn}$.
In this section we prove Theorem~\ref{thm:harmonic-projection}, which shows that when $f$ depends on $(1-\epsilon)\log n$ variables, it is close to its harmonic projection under the measure $\mu_p$. Together with Corollary~\ref{cor:l2-invariance-harmonic}, this allows us to deduce properties of $f$ on the slice given properties of $f$ on the Boolean cube, an idea formalized in Corollary~\ref{cor:harmonic-projection}.

We start by examining single monomials.

\begin{lemma} \label{lem:monomial-invariance}
 Let $m$ be a monomial of degree $d$, and let $f$ be the unique harmonic multilinear polynomial agreeing with $m$ on $\binom{[n]}{k}$ (where $d \leq k \leq n/2$).
 Then $\deg f = d$ and the coefficient $c_m$ of $m$ in $f$ is
\[ c_m = \frac{n-2d+1}{n-d+1} = 1 - O\left(\frac{d}{n}\right). \]
\end{lemma}
\begin{proof}
 Without loss of generality we can assume that $m = x_{n-d+1} \cdots x_n$. Let $B = \{n-d+1,\ldots,n\}$. Recall that the basis element $\chi_B$ is equal to
\[
 \chi_B = \sum_{a_1 \neq \cdots \neq a_d \in [n-d]} (x_{a_1} - x_{n-d+1}) \cdots (x_{a_d} - x_n).
\]

 Let $f$ be the unique harmonic multilinear polynomial agreeing with $m$ on $\binom{[n]}{k}$. Corollary~\ref{cor:slice-degree} shows that $\deg f \leq d$.
 The coefficient $\hat{f}(B)$ of $\chi_B$ in the Fourier expansion of $f$ is given by the formula $\hat{f}(B) = \langle f, \chi_B \rangle / \|\chi_B\|^2$. Since $\deg f \leq d$, it is not hard to check that in the Fourier expansion of $f$, the monomial $m$ only appears in $\chi_B$.
 Therefore the coefficient $c_m$ of $m$ in $f$ is
\[
 c_m = (-1)^d (n-d)^{\underline{d}} \frac{\langle f,\chi_B \rangle}{\|\chi_B\|^2},
\]
 since there are $(n-d)^{\underline{d}}$ summands in the definition of $\chi_B$. The value of $\|\chi_B\|^2$ is given by Lemma~\ref{lem:harmonic-expansion} and Lemma~\ref{lem:harmonic-basis-norms}:
\[
 \|\chi_B\|^2 = \binom{n-d+1}{2} \binom{n-d}{2} \cdots \binom{n-2d+2}{2} 2^d \frac{k^{\underline{d}} (n-k)^{\underline{d}}}{n^{\underline{2d}}}.
\]
 We proceed to compute $\langle f,\chi_B \rangle$. Let $S \in \binom{[n]}{k}$. If $f(S) \chi_B(S) \neq 0$ then $B \subseteq S$, which happens with probability $k^{\underline{d}}/n^{\underline{d}}$. The number of non-zero terms (each equal to $(-1)^d$) is the number of choices of $a_1,\ldots,a_d \notin S$, namely $(n-k)^{\underline{d}}$. Therefore $\langle f,\chi_B \rangle = (-1)^d k^{\underline{d}} (n-k)^{\underline{d}}/n^{\underline{d}}$, and so
\begin{align*}
 \!c_m &= (n-d)^{\underline{d}} \cdot \frac{k^{\underline{d}} (n-k)^{\underline{d}}}{n^{\underline{d}}} \cdot \frac{n^{\underline{2d}}}{(n-d+1)(n-d)^2\cdots(n-2d+2)^2(n-2d+1) k^{\underline{d}} (n-k)^{\underline{d}}} \\ &=
 \frac{(n-d)^{\underline{d}} (n-d)^{\underline{d}}}{(n-d+1)(n-d)^2\cdots(n-2d+2)^2(n-2d+1)} \\ &=
 \frac{n-2d+1}{n-d+1}.
\end{align*}
 Finally, since $c_m \neq 0$ and $\deg f \leq d$, we can conclude that $\deg f = d$.
\end{proof}

As a consequence, we obtain a result on Fourier characters on the cube.

\begin{lemma} \label{lem:character-invariance}
 Let $\omega = \omega_S$ be a Fourier character with respect to the measure $\mu_p$ of degree $d$, and let $\tilde\omega$ be the unique harmonic multilinear polynomial agreeing with $\omega$ on $\binom{[n]}{np}$ (where $d \leq np \leq n/2$). For $\pi \in \{\mu_p,\cG_p\}$ we have
\[
 \|\omega - \tilde\omega\|_\pi^2 = O\left(\frac{d^2}{p(1-p)n}\right).
\]
\end{lemma}
\begin{proof}
 Recall that
\[
 \omega = \frac{1}{(p(1-p))^{d/2}} \prod_{i \in S} (x_i - p).
\]
 Lemma~\ref{lem:monomial-invariance} shows that
\[
 \tilde\omega = \frac{c}{(p(1-p))^{d/2}} \prod_{i \in S} x_i + \eta, \quad c = 1 - O\left(\frac{d}{n}\right),
\]
 where $\eta$ involves other monomials. In fact, since $\tilde\omega$ is harmonic, it is invariant under shifting all the variables by $p$, and so
\[
 \tilde\omega = c\omega + \eta',
\]
 where $\eta'$ involves other characters. Due to orthogonality of characters we have
\[
 \|\tilde\omega - \omega\|_\pi^2 = \|\tilde\omega\|_\pi^2 - (2c-1)\|\omega\|_\pi^2 = \|\tilde\omega\|_\pi^2 - (2c-1) = \|\tilde\omega\|_\pi^2 - 1 + O\left(\frac{d}{n}\right).
\]

 Since $\tilde\omega$ is harmonic, Corollary~\ref{cor:l2-invariance-harmonic} allows us to estimate $\|\tilde\omega\|_\pi^2$ given $\|\tilde\omega\|_{\nu_{pn}}^2$, which we proceed to estimate:
\begin{align*}
 \|\tilde\omega\|_{\nu_{pn}}^2 &= \frac{1}{(p(1-p))^d} \sum_{t=0}^d \binom{d}{t} \frac{k^{\underline{t}}(n-k)^{\underline{d-t}}}{n^{\underline{d}}} (1-p)^{2t} p^{2(d-t)} \\ &=
 \frac{1}{(p(1-p))^d} \sum_{t=0}^d \binom{d}{t} p^t (1-p)^{d-t} (1-p)^{2t} p^{2(d-t)} \left(1 \pm O\left(\frac{d^2}{p(1-p)n}\right)\right) \\ &=
 1 \pm O\left(\frac{d^2}{p(1-p)n}\right).
\end{align*}
 Corollary~\ref{cor:l2-invariance-harmonic} shows that the same estimate holds even with respect to $\pi$, and so
\[
 \|\tilde\omega - \omega\|_\pi^2 = \|\tilde\omega\|_\pi^2 - 1 + O\left(\frac{d^2}{p(1-p)n}\right) = O\left(\frac{d^2}{p(1-p)n}\right). \qedhere
\]
\end{proof}

We can now conclude that a multilinear polynomial depending on a small number of variables is close to its harmonic projection.

\begin{theorem} \label{thm:harmonic-projection}
 Let $f$ be a multilinear polynomial depending on $d$ variables, and let $\tilde{f}$ be the unique harmonic multilinear polynomial agreeing with $f$ on $\binom{[n]}{pn}$, where $d \leq pn \leq n/2$. For $\pi \in \{\mu_p,\cG_p\}$ we have
\[
 \|f - \tilde{f}\|_\pi^2 = O\left(\frac{d^22^d}{p(1-p)n}\right)\|f\|_\pi^2.
\]
\end{theorem}
\begin{proof}
 We can assume without loss of generality that $f$ depends on the first $d$ coordinates.
 Express $f$ as a linear combination of characters: $f = \sum_{S \subseteq [d]} \hat{f}(S) \omega_S$. Clearly $\tilde{f} = \sum_{S \subseteq [d]} \hat{f}(S) \tilde{\omega}_S$, where $\tilde{\omega}_S$ is the unique function agreeing with $\omega_S$ on $\binom{[n]}{pn}$. Lemma~\ref{lem:character-invariance} together with the Cauchy--Schwartz inequality shows that
\[
 \|f-\tilde{f}\|_\pi^2 \leq 2^d \sum_{S \subseteq [d]} \hat{f}(S)^2 O\left(\frac{d^2}{p(1-p)n}\right) = O\left(\frac{d^2 2^d}{p(1-p)n}\right) \|f\|_{\pi}^2.
\]
 This completes the proof.
\end{proof}

Combining Theorem~\ref{thm:harmonic-projection} with Corollary~\ref{cor:l2-invariance-harmonic}, we show how to deduce properties of $f$ on the slice given its properties on the cube.

\begin{corollary} \label{cor:harmonic-projection}
 Let $f$ be a multilinear polynomial depending on $d$ variables, and let $\tilde{f}$ be the unique harmonic multilinear polynomial agreeing with $f$ on $\binom{[n]}{pn}$, where $d \leq pn \leq n/2$. Suppose that $\|f\|_{\mu_p}^2 = \|f\|_{\cG_p}^2 \leq 1$. For $\pi \in \{\mu_p,\cG_p\}$ we have:
\begin{enumerate}
 \item $|\EE_\pi[f] - \EE_{\nu_{pn}}[\tilde{f}]| = O_p(\frac{d 2^{d/2}}{\sqrt{n}})$.
 \item $\|\tilde{f}\|_{\nu_{pn}} = 1 \pm O_p(\frac{d 2^{d/2}}{\sqrt{n}})$.
 \item $\VV[\tilde{f}]_{\nu_{pn}} = \VV[f]_\pi \pm O_p(\frac{d 2^{d/2}}{\sqrt{n}})$.
 \item For all $\rho \in [0,1]$, $\stabilityc_\rho[\tilde{f}]_{\nu_{pn}} = \stabilityc_\rho[f]_\pi \pm O_p(\frac{d 2^{d/2}}{\sqrt{n}})$.
 \item For all $\ell \leq d$, $\|\tilde{f}^{=\ell}\|_{\nu_{pn}} = \|f^{=\ell}\|_{\mu_p} \pm O_p(\frac{d 2^{d/2}}{\sqrt{n}})$.
\end{enumerate}
\end{corollary}
\begin{proof}
 Throughout the proof, we are using Corollary~\ref{cor:l2-invariance-harmonic} to convert information on $\tilde{f}$ with respect to $\pi$ to information on $\tilde{f}$ with respect to $\nu_{pn}$. All calculations below are with respect to $\pi$.

 For the first item, note that
\[
 |\EE[f] - \EE[\tilde{f}]| \leq \|f-\tilde{f}\|_1 \leq \|f-\tilde{f}\|_2 = O_p\left(\frac{d 2^{d/2}}{\sqrt{n}}\right).
\]

 The second item follows from the triangle inequality
\[
 \|f\| - \|f - \tilde{f}\| \leq \|\tilde{f}\| \leq \|f\| + \|f - \tilde{f}\|.
\]

 For the third item, notice first that $|\EE[f]| \leq \|f\|_1 \leq \|f\|_2 = 1$. The item now follows from the previous two.

 The fourth item follows from the fact that $\stabilityc_\rho$ is $1$-Lipschitz, which in turn follows from the fact that $\stabilityc_\rho[f] = \|T_{\sqrt{\rho}}f\|^2$ and that $T_{\sqrt{\rho}}$ is a contraction.

 For the fifth item, assume that $f$ depends on the first $d$ variables, and write $f = \sum_{S \subseteq [d]} c_S \omega_S$. We have
\[
 \tilde{f}^{=\ell} = \sum_{S \subseteq [d]} c_S (\tilde{\omega}_S)^{=\ell} = \widetilde{f^{=\ell}} + \sum_{|S|>\ell} c_S (\tilde{\omega}_S)^{=\ell}.
\]
 Lemma~\ref{lem:character-invariance} shows that for $|S| > \ell$, $\|(\tilde{\omega}_S)^{=\ell}\|^2 \leq \|\omega_S - \tilde{\omega}_S\|^2 = O_p(\frac{d^2}{n})$. Therefore
\[
 \|\tilde{f}^{=\ell} - \widetilde{f^{=\ell}}\|^2 \leq 2^d \sum_{|S|>\ell} c_S^2 O_p\left(\frac{d}{n}\right) = O_p\left(\frac{d^22^d}{n}\right).
\]
 The fifth item now follows from the triangle inequality and the second item.
\end{proof}

\section{Invariance principle} \label{sec:invariance}

In the sequel, we assume that parameters $p \in (0,1/2]$ and $n$ such that $pn$ is an integer are given.
The assumption $p \leq 1/2$ is without loss of generality.

We will use big O notation in the following way: $f = O_p(g)$ if for all $n \geq N(p)$, it holds that $f \leq C(p) g$, where $N(p),C(p)$ are continuous in $p$. In particular, for any choice of $p_L,p_H$ satisfying $0 < p_L\leq p_H < 1$, if $p \in [p_L,p_H]$ then $f = O(g)$. Stated differently, as long as $\lambda \leq p \leq 1-\lambda$, we have a uniform estimate $f = O_\lambda(g)$.
Similarly, all constants depending on $p$ (they will be of the form $A_p$ for various letters $A$) depend continuously on $p$.

\paragraph*{Proof sketch} Let $\psi$ be a Lipschitz functional and $f$ a harmonic multilinear polynomial of unit variance, low slice-influences, and low degree $d$. A simple argument shows that $f$ also has low cube-influences, and this implies that
\[
 \EE_{\nu_k}[\psi(f)] \approx \EE_{\nu_{pn}}[\psi(f)] \pm O_p\left(\frac{|k-np|}{\sqrt{n}} \cdot \sqrt{d}\right).
\]
The idea is now to apply the multidimensional invariance principle jointly to $f$ and to $S = \frac{x_1+\cdots+x_n-np}{\sqrt{p(1-p)n}}$, deducing
\[
 \EE_{\mu_p} [\psi(f) \charf{|S| \leq \sigma}] = \EE_{\cG_p} [\psi(f) \charf{|S| \leq \sigma}] \pm \epsilon.
\]
An application of Lemma~\ref{lem:harmonic-gauss} shows that
\[
 \EE_{\cG_p} [\psi(f) \charf{|S| \leq \sigma}] = \Pr_{\cG_p}[|S| \leq \sigma] \EE_{\cG_p} [\psi(f)].
\]
Similarly,
\[
 \EE_{\mu_p}[\psi(f) \charf{|S| \leq \sigma}] = \Pr_{\mu_p}[|S| \leq \sigma] (\EE_{\mu_p} [\psi(f)] \pm O_p(\sigma \sqrt{d})).
\]
Since $\Pr_{\cG_p}[|S| \leq \sigma] \approx \Pr_{\mu_p}[|S| \leq \sigma] = \Theta_p(\sigma)$, we can conclude that
\[
 \EE_{\nu_{pn}} [\psi(f)] \approx \EE_{\cG_p} [\psi(f)] \pm O_p\left(\sigma \sqrt{d} + \frac{\epsilon}{\sigma}\right).
\]
By choosing $\sigma$ appropriately, we balance the two errors and obtain our invariance principle.

For minor technical reasons, instead of using $\charf{|S| \leq \sigma}$ we actually use a Lipschitz function supported on $|S| \leq \sigma$.

\paragraph*{Main theorems} Our main theorem is Theorem~\ref{thm:invariance}, proved in Section~\ref{sec:invariance-main} on page~\pageref{thm:invariance}. This is an invariance principle for low-degree, low-influence functions and Lipschitz functionals, comparing the uniform measure on the slice $\nu_{pn}$ to the measure $\mu_p$ on the Boolean cube and to the Gaussian measure $\cG_p$.

Some corollaries appear in Section~\ref{sec:invariance-cor} on page~\pageref{sec:invariance-cor}. Corollary~\ref{cor:levy-distance} gives a bound on the L\'evy distance between the distributions $f(\nu_{pn})$ and $f(\cG_p)$ for low-degree, low-influences functions. Corollary~\ref{cor:cdf-distance} gives a bound on the CDF distance between the distributions $f(\nu_{pn})$ and $f(\cG_p)$ for low-degree, low-influences functions. Corollary~\ref{cor:invariance-noise} extends the invariance principle to functions of arbitrary degree to which a small amount of noise has been applied.

\subsection{Main argument} \label{sec:invariance-main}

We start by showing that from the point of view of L2 quantities, distributions similar to $\mu_p$ behave similarly.

\begin{definition} \label{def:like-distribution}
Let $p \in (0,1)$. A parameter $q \in (0,1)$ is \emph{$p$-like} if $|p-q| \leq \sqrt{p(1-p)/n}$.
A distribution is \emph{$p$-like} if it is one of the following:
$\mu_q,\nu_{qn},\cG_q$, where $q$ is $p$-like.
\end{definition}

\begin{lemma} \label{lem:l2-invariance}
Let $f$ be a harmonic multilinear polynomial of degree $d \leq \sqrt{n}$, and let $\pi_1,\pi_2$ be two $p$-like distributions. Then
\[ \|f\|^2_{\pi_1} = \|f\|^2_{\pi_2} \left(1 \pm O_p\left(d/\sqrt{n}\right)\right). \]
The same holds if we replace $\|f\|^2$ with $\Infs_{ij}[f] = \|f-f^{(ij)}\|^2$ or $\Infc_i[f] = \|\frac{\partial f}{\partial x_i}\|^2$.

Furthermore, there is a constant $S_p$ such that if $d \leq S_p \sqrt{n}$ then for all $p$-like distributions $\pi_1,\pi_2$,
\[ \frac{1}{2} \leq \frac{\|f\|^2_{\pi_1}}{\|f\|^2_{\pi_2}} \leq 2. \]
\end{lemma}
\begin{proof}
Let $\alpha_D(q) = (2q(1-q))^D$, where $D \leq d$. An easy calculation shows that $\alpha'_D(q) = 2(1-2q)D(2q(1-q))^{D-1}$, and in particular $|\alpha'_D(q)| = O(D\alpha_D(q)/q(1-q))$. It follows that for $p$-like $q$, $\alpha_D(q) = \alpha_D(p)(1 \pm O_p(D/\sqrt{n}))$. Lemma~\ref{lem:harmonic-basis-norms} thus shows that for $\pi \in \{\mu_q,\nu_{qn},\cG_q\}$ and all $D \leq d$,
\[
 \|\chi_D\|_\pi^2 = \alpha_D(q) \left(1 \pm O_p\left(\frac{D^2}{n}\right)\right) = (2p(1-p))^D \left(1 \pm O_p\left(\frac{D}{\sqrt{n}} + \frac{D^2}{n}\right)\right).
\]
Since $D \leq d$ and $d \leq \sqrt{n}$ implies $d^2/n \leq d/\sqrt{n}$, we conclude that
\[
 \|\chi_D\|_\pi^2 = (2p(1-p))^D \left(1 \pm O_p\left(\frac{d}{\sqrt{n}}\right)\right).
\]
The lemma now follows from Corollary~\ref{cor:norms}.
\end{proof}

We single out polynomials whose degree satisfies $d \leq S_p \sqrt{n}$.

\begin{definition} \label{def:low-degree}
A polynomial has \emph{low degree} if its degree is at most $S_p\sqrt{n}$, where $S_p$ is the constant in Lemma~\ref{lem:l2-invariance}.
\end{definition}

We can bound the cube-influence of a harmonic multilinear polynomial in terms of its slice-influence.

\begin{lemma} \label{lem:derivative-norm}
Let $f$ be a harmonic multilinear polynomial of low degree $d$, and let $\pi$ be a $p$-like distribution.
For all $i \in [n]$, with respect to $\pi$:
\[
 \Infc_i[f] \leq
 O_p\left(\frac{d}{n}\VV[f] + \Infs_i[f]\right).
\]
\end{lemma}
\begin{proof}
We will show that for the product measure $\pi = \mu_p$ it holds that
\[
 \Infc_i[f] \leq \frac{2d}{p(1-p)(n-d)} \VV[f] + \frac{2n}{p(1-p)(n-d)} \Infs_i[f]
 \]
 which will imply the statement of the lemma by
 Lemma~\ref{lem:l2-invariance}.

 The idea is to come up with an explicit expression for $\Infs_i[f]$. Let $j\neq i$. For $S$ not containing $i,j$ we have
\[
 \omega_S^{(ij)} = \omega_S, \quad
 \omega_{S\cup\{i\}}^{(ij)} = \omega_{S\cup\{j\}}, \quad
 \omega_{S\cup\{j\}}^{(ij)} = \omega_{S\cup\{i\}}, \quad
 \omega_{S\cup\{i,j\}}^{(ij)} = \omega_{S\cup\{i,j\}}.
\]
 Therefore
\[
 \Infs_{ij}[f] = \|f-f^{(ij)}\|^2 = \sum_{i,j \notin S} (\hat{f}(S\cup\{i\}) - \hat{f}(S\cup\{j\}))^2.
\]
 On the other hand, we have
\[
 p(1-p) \Infc_i[f] = \sum_{S\ni i} \hat{f}(S)^2 \leq \frac{1}{n-d} \sum_{S\ni i} \hat{f}(S)^2 (n-|S|) = \frac{1}{n-d} \sum_{j\neq i} \sum_{i,j \notin S} \hat{f}(S\cup\{i\})^2.
\]
 The L2 triangle inequality shows that $\hat{f}(S\cup\{i\})^2 \leq 2\hat{f}(S\cup\{j\})^2 + 2(\hat{f}(S\cup\{i\}) - \hat{f}(S\cup\{j\}))^2$, and so
\begin{align*}
 p(1-p) \Infc_i[f] &\leq
 \frac{2}{n-d} \sum_{j\neq i} \sum_{i,j \notin S} \hat{f}(S\cup\{j\})^2 + \frac{2}{n-d} \sum_{j\neq i} \Infs_{ij}[f] \\ &\leq
 \frac{2p(1-p)}{n-d} \sum_{j\neq i} \Infc_j[f] + \frac{2n}{n-d} \Infs_i[f] \\ &\leq
 \frac{2d}{n-d} \VV[f] + \frac{2n}{n-d} \Infs_i[f],
\end{align*}
 using the Poincar\'e inequality.
 Rearranging, we obtain the statement of the lemma.
\end{proof}

Using Lemma~\ref{lem:derivative-norm}, we can show that the behavior of a low degree function isn't too sensitive to the value of $q$ in $\nu_{qn}$.

\begin{lemma} \label{lem:similar-slices}
Let $f$ be a harmonic multilinear polynomial of low degree $d$, and let $\ell$ be an integer such that $\nu_\ell$ is $p$-like.
For every $C$-Lipschitz functional $\psi$,
\[ |\EE_{x\sim\nu_\ell}[\psi(f(x))] - \EE_{x\sim\nu_{\ell+1}}[\psi(f(x))]| = O_p\left(C\sqrt{\frac{d}{n}\VV[f]_{\nu_{pn}}}\right). \]
\end{lemma}
\begin{proof}
Let $q = \ell/n$, which is $p$-like.
For $i \in [n]$, let $(X^i,Y^i)$ be the distribution obtained by choosing a random $X^i \in \binom{[n] \setminus \{i\}}{\ell}$ and setting $Y^i = X^i \cup \{i\}$.
Note that $f(X^i) - f(Y^i) = (f-f^{[i]})(X^i)$.
Since $\Pr_{\nu_\ell}[x_i = 0] = 1-q$, we have
\[
\EE[(f(X^i) - f(Y^i))^2] \leq (1-q)^{-1} \Infc_i[f]_{\nu_\ell} = O_p\left(\frac{d}{n}\VV[f]_{\nu_{pn}} + \Infs_i[f]_{\nu_{pn}}\right),
\]
using Lemma~\ref{lem:derivative-norm} and Lemma~\ref{lem:l2-invariance}.

Consider now the distribution $(X,Y)$ supported on $\binom{[n]}{\ell} \times \binom{[n]}{\ell+1}$ obtained by taking $X\sim\nu_\ell$ and choosing $Y \supset X$ uniformly among the $n-\ell$ choices; note that $Y\sim\nu_{\ell+1}$. Since $(X,Y)$ is a uniform mixture of the distributions $(X^i,Y^i)$, we deduce
\begin{align*}
\EE[(f(X) - f(Y))^2] &\leq  O_p\left(\frac{d}{n}\VV[f]_{\nu_{pn}} + \frac{1}{n} \Infs[f]_{\nu_{pn}}\right) \\ &\leq
O_p\left(\frac{d}{n} \VV[f]_{\nu_{pn}}\right),
\end{align*}
using the Poincar\'e inequality $\Infs[f] \leq d\VV[f]$ (see Section~\ref{sec:analysis}). The lemma now follows along the lines of Lemma~\ref{lem:lipschitz-cs}.
\end{proof}

We now apply a variant of the invariance principle for Lipschitz functionals due to Isaksson and Mossel.

\begin{proposition}[{\cite[Theorem 3.4]{IM}}] \label{pro:invariance-IM}
 Let $Q_1,\ldots,Q_k$ be $n$-variate multilinear polynomials of degree at most $d$ such that with respect to $\mu_p$, $\VV[F_i] \leq 1$ and $\Infc_j[F_i] \leq \tau$ for all $i \in [k]$ and $j \in [n]$. For any $C$-Lipschitz functional $\Psi\colon\RR^k\to\RR$ (i.e., a function satisfying $|\Psi(x)-\Psi(y)| \leq C\|x-y\|_2$),
\[ |\EE_{\mu_p}[\Psi(Q_1,\ldots,Q_k)] - \EE_{\cG_p}[\Psi(Q_1,\ldots,Q_k)]| = O_k(C\rho_p^d\tau^{1/6}), \]
 for some (explicit) constant $\rho_p \geq 1$.
\end{proposition}

\begin{lemma} \label{lem:invariance}
 Denote
\[ S = \frac{\sum_{i=1}^n x_i-np}{\sqrt{p(1-p)n}}. \]
 Let $f$ be a harmonic multilinear polynomial of low degree $d \geq 1$ such that with respect to $\mu_p$, $\EE[f] = 0$, $\VV[f] \leq 1$ and $\Infs_i[f] \leq \tau$ for all $i \in [n]$. Suppose that $\tau \leq R_p^{-d}$ and $n \geq R_p^d$, for some constant $R_p$.
 For any $C$-Lipschitz functional $\psi$ such that $\psi(0) = 0$ and $B$-Lipschitz functional $\phi$ (where $B \geq 1$) satisfying $\|\phi\|_\infty \leq 1$,
\[
 |\EE_{\mu_p}[\psi(f)\phi(S)] - \EE_{\cG_p}[\psi(f)\phi(S)]| = CO_p\left(\sqrt{B}\rho_p^{d/2}\left(\tau + \frac{d}{n}\right)^{1/12}\right).
\]
 The condition $\Infs_i[f] \leq \tau$ for all $i \in [n]$ can be replaced by the condition $\Infc_i[f]_{\mu_p} \leq \tau$ for all $i \in [n]$.
\end{lemma}
\begin{proof}
 For $M$ to be chosen later, define
\[
 \tilde\psi(x) = \begin{cases} -M & \text{if } \psi(x) \leq -M, \\ \psi(x) & \text{if } -M \leq \psi(x) \leq M, \\ M & \text{if } M \leq \psi(x). \end{cases}
\]
 It is not hard to check that $\tilde\psi$ is also $C$-Lipschitz.

 We are going to apply Proposition~\ref{pro:invariance-IM} with $Q_1 = f$, $Q_2 = S/\sqrt{p(1-p)n}$, and $\Psi(y_1,y_2) = \tilde\psi(y_1) \phi(y_2)$.
 With respect to $\mu_p$, $\VV[Q_2] = 1$ and $\Infc_i[Q_2] = 1/(p(1-p)n)$ for all $i \in [n]$.  Lemma~\ref{lem:derivative-norm} shows that $\Infc_i[f] = O_p(\frac{d}{n} + \tau)$, and so the cube-influences of $Q_1,Q_2$ are bounded by $O_p(\tau + \frac{d}{n})$.
 Since
\begin{multline*}
 |\Psi(y_1,y_2) - \Psi(z_1,z_2)| \leq |\Psi(y_1,y_2) - \Psi(y_1,z_2)| + |\Psi(y_1,z_2) - \Psi(z_1,z_2)| \leq \\ MB|y_2-z_2|+C|y_1-z_1|,
\end{multline*}
 we see that $\Psi$ is $(MB+C)$-Lipschitz. Therefore
\[
 |\EE_{\mu_p}[\tilde\psi(f)\phi(S)] - \EE_{\cG_p}[\tilde\psi(f)\phi(S)]| = O_p\left((MB+C)\rho_p^d\left(\tau + \frac{d}{n}\right)^{1/6}\right).
\]
 Next, we want to replace $\tilde\psi$ with $\psi$. For $\pi \in \{\mu_p,\cG_p\}$ we have
\begin{multline*}
 |\EE_\pi[\tilde\psi(f)\phi(S)] - \EE_\pi[\psi(f)\phi(S)]| \leq
 \EE_\pi[|\psi(f)|\,|\phi(S)|\charf{|\psi(f)| \geq M}] \leq \\
 C\EE_\pi[|f|\charf{|f| \geq M/C}] \leq
 \frac{C^2}{M}\EE_\pi[|f|^2] \leq \frac{C^2}{M}.
\end{multline*}
 Therefore
\[
 |\EE_{\mu_p}[\psi(f)\phi(S)] - \EE_{\cG_p}[\psi(f)\phi(S)]| = O_p\left((MB+C)\rho_p^d\left(\tau + \frac{d}{n}\right)^{1/6} + \frac{C^2}{M}\right).
\]
 Choosing $M = C/\sqrt{B\rho_p^d\left(\tau + \frac{d}{n}\right)^{1/6}}$ completes the proof. The conditions on $\tau,n$ guarantee that $\rho_p^d(\tau+\frac{d}{n})^{1/6} \leq 1$, and so $B \geq 1$ allows us to obtain the stated error bound.
\end{proof}

In order to finish the proof, we combine Lemma~\ref{lem:invariance} with Lemma~\ref{lem:similar-slices}.

\begin{theorem} \label{thm:invariance}
 Let $f$ be a harmonic multilinear polynomial of degree $d$ such that with respect to $\nu_{pn}$, $\VV[f] \leq 1$ and $\Infs_i[f] \leq \tau$ for all $i \in [n]$.
 Suppose that $\tau \leq I_p^{-d}\delta^K$ and $n \geq I_p^d/\delta^K$, for some constants $I_p,K$.
 For any $C$-Lipschitz functional $\psi$ and for $\pi \in \{\cG_p,\mu_p\}$,
\[
 |\EE_{\nu_{pn}}[\psi(f)] - \EE_{\pi}[\psi(f)]| = O_p(C\delta).
\]
 The condition $\Infs_i[f] \leq \tau$ for all $i \in [n]$ can be replaced by the condition $\Infc_i[f]_{\mu_p} \leq \tau$ for all $i \in [n]$.
\end{theorem}
\begin{proof}
 We prove the theorem for $\pi = \cG_p$. The version for $\mu_p$ then follows from the classical invariance principle, using Lemma~\ref{lem:derivative-norm}.

 Replacing $f$ with $f - \EE[f]$ (recall that the expectation of $f$ is the same with respect to both $\mu_p$ and $\pi$) doesn't change the variance and influences of $f$, so we can assume without loss of generality that $\EE[f] = 0$. Similarly, we can replace $\psi$ with $\psi - \psi(0)$ without affecting the quantity $\EE_{\nu_{pn}}[\psi(f)] - \EE_{\mu_p}[\psi(f)]$, and so we can assume without loss of generality that $\psi(0) = 0$.

 For a parameter $\sigma \leq 1$ to be chosen later, define a function $\phi$ supported on $[-\sigma,\sigma]$ by
\[
 \phi(x) =
 \begin{cases}
 1 + x/\sigma & \text{if } -\sigma \leq x \leq 0, \\
 1 - x/\sigma & \text{if } 0 \leq x \leq \sigma.
 \end{cases}
\]
 Note that $\|\phi\|_\infty = 1$ and that $\phi$ is $(1/\sigma)$-Lipschitz. Lemma~\ref{lem:invariance} (together with Lemma~\ref{lem:l2-invariance}) shows that
\begin{align*}
 |\EE_{\mu_p}[\psi(f)\phi(S)] - \EE_{\cG_p}[\psi(f)\phi(S)]| &= CO_p\left(\sigma^{-1/2}\rho_p^{d/2}\left(\tau + \frac{d}{n}\right)^{1/12}\right) \\ &=
 CO_p(\sigma^{-1/2}\rho_p^{d/2}(\tau^{1/12} + n^{-1/24})),
\end{align*}
 assuming $\tau \leq R_p^{-d}$ and $n \geq R_p^d$ (the condition on $n$ implies that $d$ is low degree).

 Let $\alpha$ be the distribution of $x_1+\cdots+x_n$ under $\cG_p$. Lemma~\ref{lem:gamma-conditioning} and Lemma~\ref{lem:harmonic-gauss} show that
 \begin{align*}
  \EE_{\cG_p}[\psi(f)\phi(S)] &= \EE_{q\sim\alpha}\big[\EE_{\gamma_{p,q}}[\psi(f)] \phi(\tfrac{q-np}{\sqrt{p(1-p)n}})\big] \\ &=
  \EE_{\cG_p}[\psi(f)] \EE_{q\sim\alpha}[\phi(\tfrac{q-np}{\sqrt{p(1-p)n}})] =
  \EE_{\cG_p}[\psi(f)] \EE_{\cG_p}[\phi(S)].
 \end{align*}
 Similarly, Lemma~\ref{lem:similar-slices} shows that
\begin{align*}
 &|\EE_{\mu_p}[\psi(f)\phi(S)] - \EE_{\nu_{np}}[\psi(f)] \EE_{\mu_p}[\phi(S)]| \\ \leq
 &\sum_{|k-np| \leq \sigma\sqrt{p(1-p)n}} \Pr_{\mu_p}[S=\tfrac{k-np}{\sqrt{p(1-p)n}}] \phi(\tfrac{k-np}{\sqrt{p(1-p)n}}) |\EE_{\nu_k}[\psi(f)] - \EE_{\nu_{pn}}[\psi(f)]| \\ \leq
 &\sum_{|k-np| \leq \sigma\sqrt{p(1-p)n}} \Pr_{\mu_p}[S=\tfrac{k-np}{\sqrt{p(1-p)n}}] \phi(\tfrac{k-np}{\sqrt{p(1-p)n}}) |k-np| O_p\left(C\sqrt{\frac{d}{n}}\right) \\ \leq
 &\EE_{\mu_p}[\phi(S)] O_p(C\sigma\sqrt{d}).
\end{align*}
 Therefore
\begin{multline*}
 |\EE_{\cG_p}[\psi(f)] \EE_{\cG_p}[\phi(S)] - \EE_{\nu_{np}}[\psi(f) \EE_{\mu_p}[\phi(S)]| \leq \\ O_p(C\sigma^{-1/2}\rho_p^{d/2}(\tau^{1/12} + n^{-1/24})) +
 O_p(C\sigma\sqrt{d} \EE_{\mu_p}[\phi(S)]).
\end{multline*}

 Proposition~\ref{pro:invariance-IM} shows that
\[
 |\EE_{\cG_p}[\phi(S)] - \EE_{\mu_p}[\phi(S)]| = O_p(\sigma^{-1/2}n^{-1/6}).
\]
 Moreover, $\EE_{\cG_p}[\psi(f)] \leq C\EE_{\cG_p}[|f|] = O_p(C)$. It follows that
\[
 \EE_{\mu_p}[\phi(S)] |\EE_{\cG_p}[\psi(f)] - \EE_{\nu_{np}}[\psi(f)]| \leq O_p(C\sigma^{-1/2}\rho_p^{d/2}(\tau^{1/12} + n^{-1/24})) + O_p(C\sigma\sqrt{d} \EE_{\mu_p}[\phi(S)]).
\]
 It is not hard to check that $\EE_{\cG_p}[\phi(S)] = \Theta_p(\sigma)$, and so for $n \geq A_p \sigma^{-9}$ we have $\EE_{\mu_p}[\phi(S)] = \Theta_p(\sigma)$, implying
\[
 |\EE_{\cG_p}[\psi(f)] - \EE_{\nu_{np}}[\psi(f)]| \leq O_p(C\sigma^{-3/2}\rho_p^{d/2}(\tau^{1/12} + n^{-1/24}) + C\sigma\sqrt{d}).
\]
 Choosing $\sigma = \rho_p^{d/5}(\tau^{1/12} + n^{-1/24})^{2/5}/d^{1/5}$, we obtain
\[
 |\EE_{\cG_p}[\psi(f)] - \EE_{\nu_{np}}[\psi(f)]| \leq O_p(C\rho_p^{d/5}(\tau^{1/30} + n^{-1/60})d^{3/10}).
\]
 It is not hard to check that if $d \leq B_p n^\beta$ and $n \geq M_p$ then $n \geq A_p \sigma^{-9}$, and that if $\tau,n^{-1} \leq \rho_p^{-\gamma d}$ then $\sigma \leq 1$; these are the conditions necessary for our estimate to hold. In fact, for an appropriate choice of $\gamma_p \geq \gamma$, the condition $n \geq \rho_p^{\gamma_p d}$ implies the condition $d \leq B_p n^\beta$, and furthermore allows us to estimate $n^{-1/60} d^{3/10} = O_p(n^{-1/70})$ (say), and to control the other error term similarly. This completes the proof of the theorem.
\end{proof}

\subsection{Corollaries} \label{sec:invariance-cor}

Theorem~\ref{thm:invariance} allows us to bound the L\'evy distance between the distribution of a low degree polynomial with respect to $\nu_{pn}$ and the distribution of the same polynomial with respect to $\cG_p$ or $\mu_p$. This is the analog of~\cite[Theorem 3.19(28)]{MOO}.

\begin{corollary} \label{cor:levy-distance}
 Let $f$ be a harmonic multilinear polynomial of degree $d$ such that with respect to $\nu_{pn}$, $\VV[f] \leq 1$ and $\Infs_i[f] \leq \tau$ for all $i \in [n]$.
 There are parameters $X_p,X$ such that for any $0 < \epsilon < 1/2$, if $\tau \leq X_p^{-d}\epsilon^X$ and $n \geq X_p^d/\epsilon^X$ then the L\'evy distance between $f(\nu_{pn})$ and $f(\pi)$ is at most $\epsilon$, for $\pi \in \{\cG_p,\mu_p\}$. In other words, for all $\sigma$,
\[
 \Pr_{\nu_{pn}}[f \leq \sigma - \epsilon] - \epsilon \leq \Pr_{\pi}[f \leq \sigma] \leq \Pr_{\nu_{pn}}[f \leq \sigma + \epsilon] + \epsilon.
\]
\end{corollary}
\begin{proof}
 Given $\sigma$ and $\epsilon$, define a function $\psi$ by
\[
 \psi(x) =
 \begin{cases}
  0 & \text{if } x \leq \sigma, \\
  \frac{x-\sigma}{\epsilon} & \text{if } \sigma \leq x \leq \sigma + \epsilon, \\
  1 & \text{if } x \geq \sigma + \epsilon.
 \end{cases}
\]
 Note that $\psi$ is $(1/\epsilon)$-Lipschitz. Theorem~\ref{thm:invariance} shows that if $\tau \leq I_p^{-d}\delta^K$ and $n \geq I_p^d/\delta^k$,
\[
 \Pr_\pi[f \leq \sigma] - \Pr_{\nu_{pn}}[f \leq \sigma + \epsilon] \leq
 \EE_\pi[\psi(f)] - \EE_{\nu_{pn}}[\psi(f)] = O_p(\delta/\epsilon).
\]
 We can similarly get a bound in the other direction. To complete the proof, choose $\delta = c_p \epsilon^2$ for an appropriate $c_p$.
\end{proof}

Using the Carbery--Wright theorem, we can bound the actual CDF distance. This is the analog of~\cite[Theorem 3.19(30)]{MOO}.

\begin{proposition}[Carbery--Wright] \label{pro:carbery-wright}
 Let $f$ be a polynomial of degree at most $d$ such that $\VV[f]_{\cG_p} = 1$. Then for all $\epsilon > 0$ and all $x$,
\[ \Pr_{\cG_p}[|f-x| \leq \epsilon] = O(d\epsilon^{1/d}). \]
\end{proposition}

\begin{corollary} \label{cor:cdf-distance}
 Let $f$ be a harmonic multilinear polynomial of degree $d$ such that with respect to $\nu_{pn}$, $\VV[f] = 1$ and $\Infs_i[f] \leq \tau$ for all $i \in [n]$.
 There are parameters $Y_p,Y$ such that for any $0 < \epsilon < 1/2$, if $\tau \leq (Y_pd)^{-d}\epsilon^{Yd}$ and $n \geq (Y_pd)^d/\epsilon^{Yd}$ then the CDF distance between $f(\nu_{pn})$ and $f(\pi)$ is at most $\epsilon$, for $\pi \in \{\cG_p,\mu_p\}$. In other words, for all $\sigma$,
\[
 |\Pr_{\nu_{pn}}[f \leq \sigma] - \Pr_{\pi}[f \leq \sigma]| \leq \epsilon.
\]
\end{corollary}
\begin{proof}
 It is enough to prove the corollary for $\pi = \cG_p$, the other case following from the corresponding result in the classical setting. Corollary~\ref{cor:levy-distance} and the Carbery--Wright theorem show that for $\tau \leq X_p^{-d}\eta^X$ and $n \geq X_p^d/\eta^X$ we have
\[
 \Pr_{\nu_{pn}}[f \leq \sigma] \leq \Pr_{\cG_p}[f \leq \sigma + \eta] + \eta \leq \Pr_{\cG_p}[f \leq \sigma] + O_p(d\eta^{1/d}).
\]
 We can similarly obtain a bound from the other direction. To complete the proof, choose $\eta = c_p (\epsilon/d)^d$ for an appropriate $c_p$.
\end{proof}

All bounds we have considered so far apply only to low degree functions. We can get around this restriction by applying a small amount of noise to the functions before applying the invariance principle itself. This is the analog of~\cite[Theorem 3.20]{MOO}.

Even though the natural noise operator to apply on the slice is $H_\rho$, from the point of view of applications it is more natural to use $U_\rho$ (which we apply syntactically). Lemma~\ref{lem:noise-operators} shows that the difference between the two noise operators is small.

\begin{corollary} \label{cor:invariance-noise}
 Let $f$ be a harmonic multilinear polynomial such that with respect to $\nu_{pn}$, $\VV[f] \leq 1$ and $\Infs_i[f] \leq \tau$ for all $i \in [n]$.
 There is a parameter $Z_p$ such that for any $0 < \epsilon < 1/2$ and $0 < \delta < 1/2$, if $\tau \leq \epsilon^{Z_p/\delta}$ and $n \geq 1/\epsilon^{Z_p/\delta}$ then for $\pi \in \{\cG_p,\mu_p\}$,
\[
 |\EE_{\nu_{pn}}[\psi(U_{1-\delta}f)] - \EE_{\pi}[\psi(U_{1-\delta}f)]| = O_p(C\epsilon).
\]
\end{corollary}
\begin{proof}
 Let $g = U_{1-\delta}f$. Let $d$ be a low degree to be decided, and split $g = g^{\leq d} + g^{>d}$. With respect to $\nu_{pn}$, $\|g^{>d}\|^2 = \sum_{t>d} (1-\delta)^{2t} \|f^{=t}\|^2 \leq (1-\delta)^{2d}$. On the other hand, Theorem~\ref{thm:invariance} shows that if $\tau \leq I_p^{-d}\epsilon^K$ and $n \geq I_p^d/\epsilon^K$ then
\[
 |\EE_{\nu_{pn}}[\psi(g^{\leq d})] - \EE_\pi[\psi(g^{\leq d})]| = O_p(C\epsilon).
\]
 Since $\|g-g^{\leq d}\| = \|g^{>d}\|$, Lemma~\ref{lem:lipschitz-cs} and Lemma~\ref{lem:l2-invariance} show that as long as the degree $d$ is low,
\[
 |\EE_{\nu_{pn}}[\psi(g)] - \EE_\pi[\psi(g)]| = O_p(C\epsilon + C(1-\delta)^d) = O_p(C\epsilon + Ce^{-\delta d}).
\]
 Choosing $d = \log(1/\epsilon)/\delta$, the resulting error is $O_p(C\epsilon)$.
 This degree is low if $\log(1/\epsilon)/\delta \leq S_p\sqrt{n}$, a condition which is implied by the stated condition on $n$.
\end{proof}

\section{Majority is stablest} \label{sec:majority}

Recall Borell's theorem.

\begin{theorem}[{Borell~\cite{Borell}}] \label{thm:borell}
 Let $f\colon \RR^n \to [0,1]$ have expectation $\mu$ with respect to $\Nor(0,1)^n$. Then $\stabilityc_\rho[f] \leq \Gamma_\rho(\mu)$, where $\Gamma_\rho(\mu)$ is the probability that two $\rho$-correlated Gaussians be at most $\Phi^{-1}(\mu)$.
\end{theorem}

Borell's theorem remains true if we replace the standard Gaussian with $\cG_p$. Indeed, given a function $f$, define a new function $g$ by $g(x) = f(\sqrt{p(1-p)}x + p)$. If $x \sim \Nor(0,1)$ then $\sqrt{p(1-p)}x+p \sim \cG_p$, and so $\EE[f]_{\cG_p} = \EE[g]_{\Nor(0,1)}$. Similarly, $\stabilityc_\rho[f]_{\cG_p} = \stabilityc_\rho[g]_{\Nor(0,1)}$. Indeed, if $y = \rho x + \sqrt{1-\rho^2}\Nor(0,1)$ then
\begin{align*}
 \sqrt{p(1-p)}y + p &= p + \rho \sqrt{p(1-p)} x + \sqrt{1-\rho^2}\Nor(0,p(1-p)) \\ &= (1-\rho) p + \rho (\sqrt{p(1-p)} x + p) + \sqrt{1-\rho^2}\Nor(0,p(1-p)).
\end{align*}
Therefore Borell's theorem for $f$ and $\cG_p$ follows from the theorem for $g$ and $\Nor(0,1)$.

\emph{Majority is stablest} states that a similar bound essentially holds for all low influence functions on the slice.
This result was originally proved using the invariance principle in~\cite{MOO}. An alternative inductive proof appears in~\cite{DMN}.

It is known (see for example~\cite{MOO}) that the bound $\Phi^{-1}(\mu)$ is achieved by threshold functions.
Corollary~\ref{cor:l2-invariance-harmonic} together with Lemma~\ref{lem:noise-operators} shows that threshold functions achieve the bound also on the slice.
Indeed, take a threshold function $f$ on $d$ variables such that with respect to $\mu_p$, $\EE[f] = \mu$ and $\stabilityc_\rho[f] \geq \Gamma_\rho(\mu) - \epsilon$. Let $\tilde{f}$ be the restriction of $f$ to the slice $\binom{[n]}{pn}$. Corollary~\ref{cor:l2-invariance-harmonic} shows that $\EE[\tilde{f}] = \mu \pm o_n(1)$ and $\stabilityc_\rho[\tilde{f}] = \stabilityc_\rho[f] \pm o_n(1)$. Lemma~\ref{lem:noise-operators} shows that $\stabilitys_\rho[\tilde{f}] = \stabilityc_\rho[f] \pm o_n(1)$. Therefore for large $n$, $\EE[\tilde{f}] \approx \mu$ and $\stabilityc_\rho[f] \geq \Gamma_\rho(\mu) - 2\epsilon$.

Our proof of majority is stablest closely follows the proof of~\cite[Theorem 4.4]{MOO} presented in~\cite[\S11.7]{Ryan}. We need an auxiliary result on $\Gamma_\rho$.

\begin{proposition}[{\cite[Lemma B.6]{MOO}}] \label{pro:gamma-lipschitz}
 For each $\rho$, the function $\Gamma_\rho$ defined in Theorem~\ref{thm:borell} is $2$-Lipschitz.
\end{proposition}

\begin{theorem} \label{thm:majority}
 Let $f\colon \binom{[n]}{pn} \to [0,1]$ have expectation $\mu$ and satisfy $\Infs_i[f] \leq \tau$ for all $i \in [n]$. For any $0 < \rho < 1$, we have
\[
 \stabilitys_\rho[f] \leq \Gamma_\rho(\mu) + O_{p,\rho}\left(\frac{\log\log\frac{1}{\alpha}}{\log\frac{1}{\alpha}}\right) + O_\rho\left(\frac{1}{n}\right), \text{ where } \alpha = \min(\tau,\tfrac{1}{n}).
\]
 The condition $\Infs_i[f] \leq \tau$ for all $i \in [n]$ can be replaced by the condition $\Infc_i[f]_{\mu_p} \leq \tau$ for all $i \in [n]$.
\end{theorem}
\begin{proof}
 We identify $f$ with the unique harmonic multilinear polynomial agreeing with it on $\binom{[n]}{pn}$.
 For a parameter $0 < \delta < 1/2$ to be chosen later, let $g = H_{1-\delta} f$. Note that the range of $g$ on $\binom{[n]}{pn}$ is included in $[0,1]$ as well, since $H_{1-\delta}$ is an averaging operator. We have
\[
 \stabilitys_\rho[f] - \stabilitys_\rho[g] = \stabilitys_\rho[f] - \stabilitys_{\rho(1-\delta)^2}[f] = \sum_{d=0}^{pn} \rho^{d(1-(d-1)/n)} (1-(1-\delta)^{2d(1-(d-1)/n)})\|f^{=d}\|^2.
\]
 Since $d \leq n/2$, we have
\[
 \rho^{d(1-(d-1)/n)} (1-(1-\delta)^{2d(1-(d-1)/n)}) \leq \rho^{d/2} (1-(1-\delta)^{2d}) \leq 2\delta d\rho^{d/2}.
\]
 The expansion $x/(1-x)^2 = \sum_d dx^d$ shows that $d\rho^{d/2} \leq \sqrt{\rho}/(1-\sqrt{\rho})^2$, and so
\begin{equation} \label{eq:majority-1}
 |\stabilityc_\rho[f] - \stabilityc_\rho[g]| \leq 2\delta \frac{\sqrt{\rho}}{(1-\sqrt{\rho})^2}.
\end{equation}
 From now on we concentrate on estimating $\stabilityc_\rho[g]$.

 Define the clumped square function $\Sq$ by
\[
 \Sq(x) = \begin{cases} 0 & \text{if } x \leq 0, \\ x^2 & \text{if } 0 \leq x \leq 1, \\ 1 & \text{if } x \geq 1. \end{cases}
\]
 It is not difficult to check that $\Sq$ is $2$-Lipschitz. Corollary~\ref{cor:invariance-noise} together with Lemma~\ref{lem:noise-operators} shows that for all $\epsilon > 0$, if $\tau,\frac{1}{n} \leq \epsilon^{Z_p/\delta}$ then
\begin{equation} \label{eq:majority-2}
 |\stabilitys_\rho[g] - \stabilityc_\rho[g]| = |\EE_{\nu_{pn}}[\Sq(H_{\sqrt{\rho}} g)] - \EE_{\cG_p}[\Sq(U_{\sqrt{\rho}} g)]| = O_p(\epsilon) + O\left(\frac{(1-\sqrt{\rho})^{-2}}{n}\right).
\end{equation}

 We would like to apply Borell's theorem in order to bound $\stabilityc_\rho[g]$, but $g$ is not necessarily bounded by $[0,1]$ on $\RR^n$. In order to handle this, we define the function $\tilde{g} = \max(0,\min(1,g))$, which is bounded by $[0,1]$. Let $\dist_{[0,1]}$ be the function which measures the distance of a point $x$ to the interval $[0,1]$. The function $\dist_{[0,1]}$ is clearly $1$-Lipschitz, and so Corollary~\ref{cor:invariance-noise} implies that under the stated assumptions on $\tau,\frac{1}{n}$, we have
\[
 \EE_{\cG_p}[|g-\tilde{g}|] = \EE_{\cG_p}[\dist_{[0,1]}(g)] = |\EE_{\nu_{pn}}[\dist_{[0,1]}(g)] - \EE_{\cG_p}[\dist_{[0,1]}(g)]| = O_p(\epsilon).
\]
 Since $U_{\sqrt{\rho}}$ is an averaging operator and $\Sq$ is $2$-Lipschitz, we conclude that
\begin{equation} \label{eq:majority-3}
 |\stabilityc_\rho[g] - \stabilityc_\rho[\tilde{g}]| = |\EE_{\cG_p}[\Sq(U_{\sqrt{\rho}}g)] - \EE_{\cG_p}[\Sq(U_{\sqrt{\rho}}\tilde{g})]| = O_p(\epsilon).
\end{equation}
 Lemma~\ref{lem:harmonic-expansion} shows that $\EE_{\cG_p}[g] = \EE_{\nu_{pn}}[g] = \EE_{\nu_{pn}}[f] = \mu$, and so $|\EE_{\cG_p}[g]-\mu| = O_p(\epsilon)$. Proposition~\ref{pro:gamma-lipschitz} implies that $\Gamma_\rho(\EE[\tilde g]) \leq \Gamma_\rho(\mu) + O_p(\epsilon)$. Applying Borell's theorem (Theorem~\ref{thm:borell}), we deduce that
\begin{equation} \label{eq:majority-4}
 \stabilityc_\rho[\tilde{g}] \leq \Gamma_\rho(\EE[\tilde g]) \leq \Gamma_\rho(\mu) + O_p(\epsilon).
\end{equation}

 Putting~\eqref{eq:majority-1},\eqref{eq:majority-2},\eqref{eq:majority-3},\eqref{eq:majority-4} together, we conclude that
\[
 \stabilitys_\rho[f] \leq \Gamma_\rho(\mu) + O_p(\epsilon) + O\left(\frac{(1-\sqrt{\rho})^{-2}}{n}\right) + 2\delta \frac{\sqrt{\rho}}{(1-\sqrt{\rho})^2}.
\]
 Taking $\delta = \epsilon$, we obtain
\[
 \stabilitys_\rho[f] \leq \Gamma_\rho(\mu) + O_{p,\rho}(\epsilon) + O_\rho\left(\frac{1}{n}\right).
\]
 The bounds on $\tau,\frac{1}{n}$ now become $\tau,\frac{1}{n} \leq \epsilon^{Z_p/\epsilon}$, from which we can extract the theorem.
\end{proof}

\section{Bourgain's theorem} \label{sec:bourgain}

Bourgain's theorem in Gaussian space gives a lower bound on the tails of Boolean functions (in this section, \emph{Boolean} means that the range of the function is $\{\pm 1\}$). We quote its version from~\cite[Theorem 2.11]{KKO}.

\begin{theorem}[Bourgain] \label{thm:bourgain}
 Let $f\colon \RR^n \to \{\pm 1\}$. For any $k \geq 1$ we have, with respect to Gaussian measure $\Nor(0,1)$,
\[ \|f^{>k}\|^2 = \Omega\left(\frac{\VV[f]}{\sqrt{k}}\right). \]
\end{theorem}
While the theorem is stated for $\Nor(0,1)$, it holds for $\cG_p$ as well. Indeed, given a function $f$, define a new function $g$ by $g(x) = f(\sqrt{p(1-p)}x + p)$. If $x \sim \Nor(0,1)$ then $\sqrt{p(1-p)}x + p \sim \cG_p$, and so $\VV[f]_{\cG_p} = \VV[g]_{\Nor(0,1)}$. Our definition of $f^{=i}$ for $\cG_p$ makes it clear that $g^{=i}(x) = f^{=i}(\sqrt{p(1-p)}x + p)$, where $g^{=i}$ is the degree $i$ homogeneous part of $g$. This implies that $\|f^{>k}\|^2_{\cG_p} = \|g^{>k}\|^2_{\Nor(0,1)}$. Therefore Bourgain's theorem for $f$ and $\cG_p$ follows from the theorem for $g$ and $\Nor(0,1)$.

Following closely the proof of~\cite[Theorem 3.1]{KKO}, we can prove a similar result for the slice.

\begin{theorem} \label{thm:bourgain-slice}
 Fix $k \geq 2$. Let $f\colon \binom{[n]}{pn} \to \{\pm 1\}$ satisfy $\Infs_i[f^{\leq k}] \leq \tau$ for all $i \in [n]$.
 For some constants $W_{p,k},C$, if $\tau \leq W_{p,k}^{-1}\VV[f]^C$ and $n \geq W_{p,k}/\VV[f]^C$ then
\[ \|f^{>k}\|^2 = \Omega\left(\frac{\VV[f]}{\sqrt{k}}\right). \]
 The condition $\Infs_i[f] \leq \tau$ for all $i \in [n]$ can be replaced by the condition $\Infc_i[f]_{\mu_p} \leq \tau$ for all $i \in [n]$.
\end{theorem}
\begin{proof}
 We treat $f$ as a harmonic multilinear polynomial.
 Since $f$ is Boolean, working over $\nu_{pn}$ we have
\[
 \|f^{>k}\|^2 = \|f-f^{\leq k}\|^2 \geq \|f^{\leq k}-\sgn(f^{\leq k})\|^2 = \|f^{\leq k}\|^2 + 1 - 2\EE[|f^{\leq k}|].
\]
 Lemma~\ref{lem:l2-invariance} shows that $\|f^{\leq k}\|^2_{\cG_p} = \|f^{\leq k}\|^2(1 \pm O_p(k^2/n))$. Since the absolute value function is $1$-Lipschitz, Theorem~\ref{thm:invariance} applied to the parameter $\delta > 0$ shows that if $\tau \leq I_p^{-k}\delta^K$ and $n \geq I_p^k/\delta^K$ then
 then $|\EE_{\nu_{pn}}[|f^{\leq k}|] - \EE_{\cG_p}[|f^{\leq k}|]| = O_p(\delta)$. This shows that
\[
 \|f^{>k}\|^2_{\nu_{pn}} \geq \|f^{\leq k}-\sgn(f^{\leq k})\|^2_{\cG_p} - O_p\left(\delta + \frac{k^2}{n}\right).
\]

 Let $g = \sgn(f^{\leq k})$. With respect to Gaussian measure $\cG_p$, $\|f^{\leq k}-g\|^2 \geq \|g^{>k}\|^2 = \Omega(\VV[g]/\sqrt{k})$, using Bourgain's theorem (Theorem~\ref{thm:bourgain}). Putting everything together, we conclude that
\begin{equation} \label{eq:bourgain-slice-1}
 \|f^{>k}\|^2_{\nu_{pn}} \geq \Omega\left(\frac{\VV[g]_{\cG_p}}{\sqrt{k}}\right) - O_p\left(\delta + \frac{k^2}{n}\right).
\end{equation}
 It remains to lower bound $\VV[g]_{\cG_p}$. Note first that over $\nu_{pn}$, $\VV[f] = 4\Pr[f=1]\Pr[f=-1]$, and so $\Pr[f=1],\Pr[f=-1] \geq \VV[f]/4$. We can furthermore assume that
\[ \Pr_{\nu_{pn}}[f^{\leq k}\geq \tfrac{2}{3}],\Pr_{\nu_{pn}}[f^{\leq k}\leq -\tfrac{2}{3}] \geq \VV_{\nu_{pn}}[f]/8, \]
 since if for example $\Pr[f^{\leq k}\geq \frac{2}{3}] \leq \VV[f]/8$ then with probability at least $\VV[f]/8$ we have $f=1$ and $f^{\leq k} < \frac{2}{3}$, and so $\|f^{>k}\|^2 = \|f-f^{\leq k}\|^2 \geq \frac{1}{9} \cdot \VV[f]/8 = \Omega(\VV[f])$. Corollary~\ref{cor:levy-distance} applied with $\epsilon = \VV[f]/16 \leq 1/3$ shows that for an appropriate $c$, if $\tau \leq X_p^{-k}/c$ and $n \geq c X_p^k$ then
\[ \Pr_{\cG_p}[f^{\leq k}\geq \tfrac{1}{3}],\Pr_{\cG_p}[f^{\leq k}\leq -\tfrac{1}{3}] \geq \VV_{\nu_{pn}}[f]/16, \]
 and so $\VV_{\cG_p}[g] \geq 4\VV_{\nu_{pn}}[f]/16(1-\VV_{\nu_{pn}}[f]/16) = \Omega(\VV_{\nu_{pn}[f]})$. Combining this with~\eqref{eq:bourgain-slice-1} shows that under $\nu_{pn}$,
\begin{equation} \label{eq:bourgain-slice-2}
 \|f^{>k}\|^2 \geq \Omega\left(\frac{\VV[f]}{\sqrt{k}}\right) - O_p\left(\delta + \frac{k^2}{n}\right).
\end{equation}
 Choosing $\delta = c_p \VV[f]/\sqrt{k}$ for an appropriate $c_p$ completes the proof.
\end{proof}

We do not attempt to match here~\cite[Theorem 3.2]{KKO}, which has the best constant in front of $\VV[f]/\sqrt{k}$.

\section{Kindler--Safra theorem} \label{sec:kindler-safra}

Theorem~\ref{thm:bourgain-slice} implies a version of the Kindler--Safra theorem~\cite{KindlerSafra,Kindler}, Theorem~\ref{thm:kindler-safra} below.

We start by proving a structure theorem for almost degree $k$ functions. We start with a hypercontractive estimate due to Lee and Yau~\cite{LeeYau} (see for example~\cite[Proposition 6.2]{F}).

\begin{proposition} \label{pro:hypercontractivity}
 For every $p$ there exists a constant $r_p$ such that for all functions $f\colon \binom{[n]}{pn} \to \RR$, $\|H_{r_p}f\|_2 \leq \|f\|_{4/3}$.
\end{proposition}

This implies the following dichotomy result.

\begin{lemma} \label{lem:dichotomy}
 Fix parameters $p$ and $k$, and let $f\colon \binom{[n]}{pn} \to \{\pm 1\}$ satisfy $\|f^{>k}\|^2 = \epsilon$. For any $i,j \in [n]$, either $\Infs_{ij}[f] \leq \epsilon/2$ or $\Infs_{i,j}[f] \geq J_{p,k}$, for some constant $J_{p,k}$.
\end{lemma}
\begin{proof}
 Let $r = r_p$ be the parameter in Proposition~\ref{pro:hypercontractivity}. Let $g = (f - f^{(ij)})/2$, so that $\Infs_{ij}[f] = 4\|g\|^2$. Since $g(x) \in \{0,\pm 1\}$, $\|g\|_{4/3}^{4/3} = \|g\|^2 = 4\Infs_{ij}[f]$. Proposition~\ref{pro:hypercontractivity} therefore implies that
\[
 (4\Infs_{ij}[f])^{3/2} = \|g\|_{4/3}^2 \geq \|H_rg\|_2^2 \geq \|H_rg^{\leq k}\|_2^2 \geq r^k\|g^{\leq k}\|^2.
\]
 Since $g = (f-f^{(ij)})/2$, we can bound $\|g^{>k}\|^2 \leq \|f^{>k}\|^2 = \epsilon$. Therefore
\[
 (4\Infs_{ij}[f])^{3/2} \geq r^k (\|g\|^2 - \epsilon) = r^k (4\Infs_{ij}[f] - \epsilon).
\]
 If $4\Infs_{ij}[f] > 2\epsilon$ then $4\Infs_{ij}[f] - \epsilon > 4\Infs_{ij}[f]/2$ and so $4\Infs_{ij}[f] \geq r^{2k}/4$.
\end{proof}

We need the following result, due to Wimmer~\cite[Proposition 5.3]{Wimmer}.

\begin{lemma}[{\cite[Proposition 5.3]{Wimmer}, \cite[Lemma 5.2]{F}}] \label{lem:slice-junta}
 Let $f\colon \binom{[n]}{pn} \to \RR$. For every $\tau > 0$ there is a set $J \subseteq [n]$ of size $O(\Infs[f]/\tau)$ such that $\Infs_{ij}[f] < \tau$ whenever $i,j \notin J$.
\end{lemma}

Combining Lemma~\ref{lem:dichotomy} and Lemma~\ref{lem:slice-junta}, we deduce that bounded degree functions depend on a constant number of coordinates, the analog of~\cite[Theorem 1]{NisanSzegedy}.

\begin{corollary} \label{cor:dichotomy}
 Fix parameters $p$ and $k$. If $f\colon \binom{[n]}{pn} \to \{\pm 1\}$ has degree $k$ then $f$ depends on $O_{p,k}(1)$ coordinates (that is, $f$ is invariant under permutations of all other coordinates).
\end{corollary}
\begin{proof}
 Apply Lemma~\ref{lem:slice-junta} with $\tau = J_{p,k}$ to obtain a set $J$ of size $O(k/J_{p,k})$. Lemma~\ref{lem:dichotomy} with $\epsilon = 0$ shows that for $i,j \notin J$ we have $\Infs_{ij}[f] = 0$, and so $f$ is invariant under permutations of coordinates outside of $J$.
\end{proof}

Using Bourgain's tail bound, we can deduce a stability version of Corollary~\ref{cor:dichotomy}, namely a Kindler--Safra theorem for the slice.

\begin{theorem} \label{thm:kindler-safra}
 Fix the parameter $k \geq 2$. Let $f\colon \binom{[n]}{pn} \to \{\pm 1\}$ satisfy $\|f^{>k}\|^2 = \epsilon$. There exists a function $h\colon \binom{[n]}{pn} \to \{\pm 1\}$ of degree $k$ depending on $O_{k,p}(1)$ coordinates (that is, invariant under permutations of all other coordinates) such that
 \[ \|f-h\|^2 = O_{p,k}\left(\epsilon^{1/C} + \frac{1}{n^{1/C}}\right), \]
 for some constant $C$.
\end{theorem}
\begin{proof}
 Let $F = f^{\leq k}$.
 We can assume that $2\epsilon < J_{p,k}/2$, since otherwise the theorem is trivial.
 Apply Lemma~\ref{lem:slice-junta} to $F$ with parameter $\tau = J_{p,k} - 2\epsilon > J_{p,k}/2$, obtaining a set $J$ of size $O(k/\tau) = O_{p,k}(1)$. It is not hard to check that
\[
 \Infs_{ij}[F] \leq \Infs_{ij}[f] \leq \Infs_{ij}[F] + 2\|f^{>k}\|^2 = \Infs_{ij}[F] + 2\epsilon.
\]
 Therefore if $i,j \notin J$ then $\Infs_{ij}[f] < \tau + 2\epsilon = J_{k,p}$, and so Lemma~\ref{lem:dichotomy} shows that $\Infs_{ij}[F] \leq \Infs_{ij}[f] = O(\epsilon)$.

 For $x \in \{0,1\}^J$, let $G_x$ and $g_x$ result from $F$ and $f$ (respectively) by restricting the coordinates in $J$ to the value $x$. It is not hard to check that $\Pr_{S \sim \nu_{pn}}[S|_J = x] \geq (p - O_p(|J|/n))^{|J|} = \Omega_{p,k}(1)$, as long as $n \geq N_{p,k}$ for some constant $N_{p,k}$; if $n \leq N_{p,k}$ then the theorem is trivial. We conclude that $\Infs_{ij}[G_x] = O_{p,k}(\epsilon)$ for all $i,j \notin J$ and $\|G_x-g_x\|^2 = \|g_x^{>k}\|^2 = O_{p,k}(\epsilon)$. Together these imply that $\Infs_{ij}[g_x] = O_{p,k}(\epsilon)$ for all $i,j \notin J$, and so $\Infs_i[g_x] = O_{p,k}(\epsilon)$ for all $i \notin J$.

 We can assume that $n - |J| \geq n/2$ (otherwise the theorem is trivial) and that the skew $p_x$ of the slice on which $G_x,g_x$ are defined satisfies $p_x = p \pm O_p(|J|/n) = \Theta(p)$, and so Theorem~\ref{thm:bourgain-slice} implies that either $\max_i \Infs_i[g_x] > W_{p,k}^{-1} \VV[g_x]^C$, or $n < 2W_{p,k}/\VV[g_x]^C$, or $\VV[g_x] = O(\sqrt{k} \|g_x^{>k}\|^2) = O_{p,k}(\epsilon)$. Since $\max_i \Infs_i[g_x] = O_{p,k}(\epsilon)$, we conclude that
\[ \VV[g_x] = O_{p,k}\left(\epsilon^{1/C} + \frac{1}{n^{1/C}}\right). \]

 Define a function $g$ by $g(S) = \EE[g_{S|_J}]$. The bound on $\VV[g_x]$ implies
\[ \|f-g\|^2 = O_{p,k}\left(\epsilon^{1/C} + \frac{1}{n^{1/C}}\right). \]
 If we let $h = \sgn g$ then we obtain the desired bound $\|f-h\|^2 \leq 4\|f-g\|^2$.

 \medskip

 It remains to show that $h$ has degree $k$ if $\epsilon$ is small enough and $n$ is large enough. We can assume without loss of generality that $J = [M]$, where $M$ is the bound on $|J|$. We have $\|f-h\|^2 \geq \|f^{>k} - h^{>k}\|^2 \geq (\|h^{>k}\| - \sqrt{\epsilon})^2$. Therefore
\[
 \|h^{>k}\| \leq \sqrt{\epsilon} + O_{p,k}\left(\epsilon^{1/2C} + \frac{1}{n^{1/2C}}\right).
\]
 On the other hand, we can write $h$ as a Boolean function $H$ of $x_1,\ldots,x_M$. Lemma~\ref{lem:monomial-invariance} shows that $\deg h \leq \deg H$, and so $\deg h > k$ implies that $\deg H > k$. Corollary~\ref{cor:harmonic-projection}(5) implies that for large enough $n$, $\|h^{>k}\| = \Omega_{p,H}(1)$. Since there are only finitely many Boolean functions on $x_1,\ldots,x_M$ which can play the role of $H$, we conclude that if $\epsilon$ is small enough and $n$ is large enough then $\deg h \leq k$.
\end{proof}

We conjecture that Theorem~\ref{thm:kindler-safra} holds with an error bound of $O_{p,k}(\epsilon)$ rather than $O_{p,k}(\epsilon^{1/C} + 1/n^{1/C})$.

\section{\texorpdfstring{$t$}{t}-Intersecting families} \label{sec:t-intersecting}

As an application of Theorem~\ref{thm:kindler-safra}, we prove a stability result for the $t$-intersecting Erd\H{o}s--Ko--Rado theorem, along the lines of Friedgut~\cite{Friedgut}. We start by stating the $t$-intersecting Erd\H{o}s--Ko--Rado theorem, which was first proved by Wilson~\cite{Wilson}.

\begin{theorem}[{\cite{Wilson}}] \label{thm:tEKR}
Let $t \geq 1$, $k \geq t$, and $n \geq (t+1)(k-t+1)$. Suppose that the family $\cF \subseteq \binom{[n]}{k}$ is $t$-intersecting: every two sets in $\cF$ have at least $t$ points in common. Then:
\begin{enumerate}[(a)]
\item $|\cF| \leq \binom{n-t}{k-t}$.
\item If $n > (t+1)(k-t+1)$ and $|\cF| = \binom{n-t}{k-t}$ then $\cF$ is a $t$-star: a family of the form
\[ \cF = \{ A \in \binom{[n]}{k} : S \subseteq A \}, \quad |S| = t. \]
\item If $t \geq 2$, $n = (t+1)(k-t+1)$ and $|\cF| = \binom{n-t}{k-t}$ then $\cF$ is either a $t$-star or a $(t,1)$-Frankl family:
\[ \cF = \{ A \in \binom{[n]}{k} : |A \cap S| \geq t+1 \}, \quad |S| = t+2. \]
\end{enumerate}
\end{theorem}

The case $t = 1$ is the original Erd\H{o}s--Ko--Rado theorem~\cite{EKR}.
Ahlswede and Khachatrian~\cite{AK2,AK3} found the optimal $t$-intersecting families for all values of $n,k,t$.

A stability version of Theorem~\ref{thm:tEKR} would state that if $|\cF| \approx \binom{n-t}{k-t}$ then $\cF$ is close to a $t$-star. Frankl~\cite{Frankl87} proved an optimal such result for the case $t = 1$. Friedgut~\cite{Friedgut} proved a stability result for all $t$ assuming that $k/n$ is bounded away from $1/(t+1)$.

\begin{theorem}[{\cite{Friedgut}}] \label{thm:tEKR-Friedgut}
 Let $t \geq 1$, $k \geq t$, $\lambda,\zeta > 0$, and $\lambda n < k < (\frac{1}{t+1} - \zeta) n$.
 Suppose $\cF \subseteq \binom{[n]}{k}$ is a $t$-intersecting family of measure $|\cF| = \binom{n-t}{k-t} - \epsilon \binom{n}{k}$. Then there exists a family $\cG$ which is a $t$-star such that
 \[ \frac{|\cF \triangle \cG|}{\binom{n}{k}} = O_{t,\lambda,\zeta}(\epsilon). \]
\end{theorem}

Careful inspection of Friedgut's proof shows that it is meaningful even for sub-constant $\zeta$, but only as long as $\zeta = \omega(1/\sqrt{n})$. We prove a stability version of Theorem~\ref{thm:tEKR} which works all the way up to $\zeta = 0$.

\begin{theorem} \label{thm:tEKR-stability}
Let $t \geq 2$, $k \geq t+1$ and $n = (t+1)(k-t+1) + r$, where $r > 0$. Suppose that $k/n \geq \lambda$ for some $\lambda > 0$.
Suppose $\cF \subseteq \binom{[n]}{k}$ is a $t$-intersecting family of measure $|\cF| = \binom{n-t}{k-t} - \epsilon \binom{n}{k}$. Then there exists a family $\cG$ which is a $t$-star or a $(t,1)$-Frankl family such that
\[ \frac{|\cF \triangle \cG|}{\binom{n}{k}} = O_{t,\lambda}\left(\max\left(\left(\frac{k}{r}\right)^{1/C},1\right)\epsilon^{1/C} + \frac{1}{n^{1/C}} \right), \]
for some constant $C$.

Furthermore, there is a constant $A_{t,\lambda}$ such that $\epsilon \leq A_{t,\lambda} \min(r/k,1)^{C+1}$ implies that $\cG$ is a $t$-star.
\end{theorem}

We do not know whether the error bound we obtain is optimal. We conjecture that Theorem~\ref{thm:tEKR-stability} holds with an error bound of $O_{t,\lambda}(\max(k/r,1)\epsilon)$.

Friedgut's approach proceeds through the $\mu_p$ version of Theorem~\ref{thm:tEKR}, first proved by Dinur and Safra~\cite{DinurSafra} as a simple consequence of the work of Ahlswede and Khachatrian. The special case $p = 1/d$ (where $d \geq 3$) also follows from earlier work of Ahlswede and Khachtrian~\cite{AK4}, who found the optimal $t$-agreeing families in $\ZZ_d^n$.

\begin{theorem}[{\cite{DinurSafra},\cite{Friedgut},\cite{FThesis}}] \label{thm:tEKR-mup}
 Let $t \geq 1$ and $p \leq 1/(t+1)$. Suppose that $\cF \subseteq \{0,1\}^n$ is $t$-intersecting. Then:
\begin{enumerate}[(a)]
\item $\mu_p(\cF) \leq p^t$~\cite{DinurSafra}.
\item If $p < 1/(t+1)$ and $\mu_p(\cF) = p^t$ then $\cF$ is a $t$-star~\cite{Friedgut}.
\item If $t \geq 2$, $p = 1/(t+1)$ and $\mu_p(\cF) = p^t$ then $\cF$ is either a $t$-star or a $(t,1)$-Frankl family~\cite{FThesis}.
\end{enumerate}
\end{theorem}

Friedgut~\cite{Friedgut} deduces his stability version of Theorem~\ref{thm:tEKR} from a stability version of Theorem~\ref{thm:tEKR-mup}. While Friedgut's stability version of Theorem~\ref{thm:tEKR-mup} is meaningful for all $p < 1/(t+1)$, his stability version of Theorem~\ref{thm:tEKR} is meaningful only for $k/n < 1/(t+1) - \omega(1/\sqrt{n})$. A more recent stability result for compressed cross-$t$-intersecting families due to Frankl, Lee, Siggers and Tokushige~\cite{FLST}, using completely different techniques, also requires $k/n$ to be bounded away from $1/(t+1)$.

Friedgut's argument combines a spectral approach essentially due to Lov\'asz~\cite{Lovasz} with the Kindler--Safra theorem~\cite{KindlerSafra,Kindler}. Using Theorem~\ref{thm:kindler-safra} instead of the Kindler--Safra theorem, we are able to obtain a stability result for the entire range of parameters of Theorem~\ref{thm:tEKR}. We restrict ourselves to the case $t \geq 2$.

Our starting point is a calculation due to Wilson~\cite{Wilson}.

\begin{theorem}[{\cite{Wilson}}] \label{thm:wilson}
 Let $t \geq 2$, $k \geq t+1$, and $n \geq (t+1)(k-t+1)$. There exists an $\binom{[n]}{k} \times \binom{[n]}{k}$ symmetric matrix $A$ such that $A_{SS} = 1$ for all $S \in \binom{[n]}{k}$, $A_{ST} = 0$ for all $S \neq T \in \binom{[n]}{k}$ satisfying $|S \cap T| \geq t$, and for all functions $f\colon \binom{[n]}{k} \to \RR$,
\begin{multline*}
Af = \sum_{e=0}^k \lambda_e f^{=e}, \\ \lambda_e = 1 + (-1)^{t-1-e} \sum_{i=0}^{t-1} (-1)^i \binom{k-1-i}{k-t} \binom{k-e}{i} \binom{n-k-e+i}{k-e} \binom{n-k-t+i}{k-t}^{-1}.
\end{multline*}

The eigenvalues $\lambda_e$ satisfy the following properties:
\begin{enumerate}[(a)]
\item $\lambda_0 = \binom{n}{k} \binom{n-t}{k-t}^{-1}$.
\item $\lambda_1 = \cdots = \lambda_t = 0$.
\item $\lambda_{t+2} \geq 0$, with equality if and only if $n = (t+1)(k-t+1)$.
\item $\lambda_{t+1} > \lambda_{t+2}$ and $\lambda_e > \lambda_{t+2}$ for $e > t+2$.
\end{enumerate}
\end{theorem}

Wilson's result actually needs $n \geq 2k$, but this is implied by our stronger assumption $k \geq t+1$ (Wilson only assumes that $k \geq t$) since $(t+1)(k-t+1)-2k = (t-1)(k-(t+1)) \geq 0$.

We need to know exact asymptotics of $\lambda_{t+2}$.

\begin{lemma} \label{lem:tEKR-ev}
 Let $t \geq 2$, $k \geq t+1$ and $n = (t+1+\rho)(k-t+1)$, where $\rho > 0$. Let $\lambda = \lambda_{t+2}$ be the quantity defined in Theorem~\ref{thm:wilson}. Then
\[
 \lambda = \Omega_t(\min(\rho,1)), \quad \lim_{\rho\to\infty} \lambda = 1.
\]
\end{lemma}
\begin{proof}
 Wilson~\cite[(4.5)]{Wilson} gives the following alternative formula for $\lambda$:
\[
 \lambda = 1 - \binom{t+1}{2} \sum_{i=0}^{t-1} \frac{2}{i+2} \binom{t-1}{i} \frac{\binom{k-t}{i+2}}{\binom{n-k-t+i}{i+2}}.
\]
 Algebraic manipulation shows that
\[
 \lambda = 1 - \sum_{i=0}^{t-1} (i+1) \binom{t+1}{i+2} \frac{\binom{k-t}{i+2}}{\binom{n-k-t+i}{i+2}}.
\]
 Calculation shows that $n-k-t = (t+\rho)(k-t+1)-2t+1$. Therefore
\[
 \lambda = 1 - \sum_{i=0}^{t-1} (i+1) \binom{t+1}{i+2} \frac{\binom{k-t}{i+2}}{\binom{(t+\rho)(k-t+1)-2t+1+i}{i+2}}.
\]
 This formula makes it clear that $\lim_{\rho\to\infty} \lambda = 1$, and that $\lambda$ is an increasing function of $\rho$.

 Assume now that $\rho \leq 1$.
 Then
\[
 \lambda = 1 - \sum_{i=0}^{t-1} (i+1) \binom{t+1}{i+2} (t+\rho)^{-i-2} \left(1 \pm O_t\left(\frac{1}{k}\right)\right).
\]
 Let us focus on the main term. Setting $\alpha = 1/(t+\rho)$, we have
\begin{align*}
 \sum_{i=0}^{t-1} (i+1) \binom{t+1}{i+2} \alpha^{i+2} &=
 \sum_{i=0}^{t-1} (i+2) \binom{t+1}{i+2} \alpha^{i+2} - \sum_{i=0}^{t-1} \binom{t+1}{i+2} \alpha^{i+2} \\ &=
 (t+1) \sum_{i=0}^{t-1} \binom{t}{i+1} \alpha^{i+2} - \sum_{i=0}^{t-1} \binom{t+1}{i+2} \alpha^{i+2} \\ &=
 (t+1) \alpha \sum_{i=1}^t \binom{t}{i} \alpha^i - \sum_{i=2}^{t+1} \binom{t+1}{i} \alpha^i \\ &=
 (t+1) \alpha ((1+\alpha)^t - 1) - ((1+\alpha)^{t+1} - 1 - (t+1)\alpha) \\ &=
 1 - (1+\alpha)^t(1-t\alpha).
\end{align*}
 Substituting $\alpha = 1/(t+\rho)$, we obtain
\[
 \sum_{i=0}^{t-1} (i+1) \binom{t+1}{i+2} (t+\rho)^{-i-2} =
 1 - \frac{(t+\rho+1)^t}{(t+\rho)^t} \frac{\rho}{t+\rho} =
 1 - \frac{\rho(t+1+\rho)^t}{(t+\rho)^{t+1}}.
\]
 Therefore when $\rho \leq 1$,
\[
 \lambda = \frac{\rho(t+1+\rho)^t}{(t+\rho)^{t+1}} \pm O_t\left(\frac{1}{k}\right).
\]
 In particular, we can find some constant $C_t$ such that
\[
 \lambda \geq \frac{(t+1+\rho)^t}{(t+\rho)^{t+1}} \left(\rho - \frac{C_t}{k}\right).
\]
 Therefore for $2C_t/k \leq \rho \leq 1$, we have $\lambda = \Omega_t(\rho)$. Since $\lambda$ is an increasing function of $\rho$, this shows that for $\rho \geq 2C_t/k$, we have $\lambda = \Omega_t(\min(\rho,1))$.

 In order to finish the proof, we handle the case $\rho \leq C_t/k$. Consider $n = (t+1)(k-t+1) + 1$.
 The value of $1 - \lambda$ in this case is
\begin{align*}
 1 - \lambda &= \sum_{i=0}^{t-1} \binom{t+1}{i+2} \binom{t-1}{i} \frac{\binom{k-t}{i+2}}{\binom{t(k-t+1)-2t+2+i}{i+2}} \\ &=
 \sum_{i=0}^{t-1} \binom{t+1}{i+2} \binom{t-1}{i} \frac{\binom{k-t}{i+2}}{\binom{t(k-t+1)-2t+1+i}{i+2}} \left(1-\frac{i+2}{t(k-t+1)-2t+2+i}\right).
\end{align*}
 The value of the last expression without the correction term $1-\frac{i+2}{t(k-t+1)-2t+2+i}$ is exactly~$1$ by Theorem~\ref{thm:wilson}, and so
\[
 \lambda \geq \frac{2}{t(k-t+1)-2t+2} = \Omega_t\left(\frac{1}{k}\right).
\]
 Since $\lambda$ is increasing in $\rho$, this shows that for all $\rho > 0$ we have $\lambda = \Omega_t(1/k)$. If also $\rho \leq C_t/k$ then this implies that $\lambda = \Omega_t(\rho)$, finishing the proof.
\end{proof}

We need a similar result comparing the measures of $t$-stars and $(t,1)$-Frankl families.

\begin{lemma} \label{lem:measure-deficit}
 Let $t \geq 2$, $k \geq t+1$ and $n = (t+1+\rho)(k-t)+t+1$, where $\rho > 0$. Let $m$ be the measure of a $t$-star, and let $m_1$ be the measure of a $(t,1)$-Frankl family. Then
\[
 \frac{m - m_1}{m} = \Omega_t(\min(\rho,1)), \quad \lim_{\rho\to\infty} \frac{m-m_1}{m} = 1.
\]
\end{lemma}
\begin{proof}
 We have
\[
 m = \binom{n-t}{k-t}, \quad m_1 = (t+2)\binom{n-t-2}{k-t-1} + \binom{n-t-2}{k-t-2}.
\]
 Computation shows that
\[
 \frac{m - m_1}{m_1} = 1 - \frac{(t+2)(n-k)(k-t) + (k-t)(k-t-1)}{(n-t)(n-t-1)}.
\]
 If $n = (t+1)(k-t+1)+r$ then calculation shows that
\[
 \frac{m - m_1}{m_1} = \frac{r(r + t(k-t) + 1)}{(n-t)(n-t-1)} \geq \frac{n-t-1}{n-t} \frac{r(r + t(k-t))}{(n-t-1)^2}.
\]
 Substituting $r = (k-t)\rho$, we obtain
\[
 \frac{m - m_1}{m_1} \geq \frac{n-t-1}{n-t} \frac{(k-t)^2\rho(\rho + t)}{(k-t)^2(t+1+\rho)^2} = \frac{n-t-1}{n-t} \frac{\rho(t + \rho)}{(t + 1 + \rho)^2}.
\]
 This shows that $\lim_{\rho\to\infty} (m-m_1)/m_1 = 1$.
 Since $n \geq (t+1)(k-t+1) \geq t+2$ implies $(n-t-1)/(n-t) \geq 1/2$, we also get
\[
 \frac{m - m_1}{m_1} \geq \frac{\rho(t + \rho)}{2(t + 1 + \rho)^2}.
\]
 As $\rho \to \infty$, the lower bound tends to $1/2$, and in particular, we can find $c_t$ such that for $\rho \geq c_t$ we have $(m - m_1)/m_1 \geq 1/3$. When $\rho \leq c_t$, we clearly have $(m - m_1)/m_1 = \Omega_t(\rho)$, completing the proof.
\end{proof}

The method of Lov\'asz~\cite{Lovasz} as refined by Friedgut~\cite{Friedgut} allows us to deduce an upper bound on $\|f^{>t}\|^2$ for the characteristic function of a $t$-intersecting family.

\begin{lemma} \label{lem:stability-fourier}
Let $t \geq 2$, $k \geq t+1$ and $n = (t+1)(k-t+1) + r$, where $r > 0$.
Let $\cF \subseteq \binom{[n]}{k}$ be a $t$-intersecting family, and $f$ its characteristic function. Then
\[
\|f^{>t}\|^2 = O\left(\max\left(\frac{k}{r},1\right)\right) \cdot (m-\EE[f]), \text{ where } m = \frac{\binom{n-t}{k-t}}{\binom{n}{k}}.
\]
\end{lemma}
\begin{proof}
Let $A$ be the matrix from Theorem~\ref{thm:wilson}. Since $\|f^{=0}\| = \EE[f]$,
\[
 \EE[f] = \langle f,Af \rangle \geq \lambda_0 \EE[f]^2 + \lambda_{t+2} \|f^{>t}\|^2.
\]
This already implies that $\EE[f] \leq \lambda_0^{-1} = m$.
Since $\lambda_0 = m^{-1}$ and $\EE[f] \leq m$, we conclude that
\[
 \|f^{>t}\|^2 \leq \frac{\EE[f]-m^{-1}\EE[f]^2}{\lambda_{t+2}} = \frac{\EE[f](1-m^{-1}\EE[f])}{\lambda_{t+2}} \leq \frac{m-\EE[f]}{\lambda_{t+2}}.
\]
Lemma~\ref{lem:tEKR-ev} completes the proof.
\end{proof}

In order to prove our stability result, we need a result on cross-intersecting families.

\begin{theorem}[{\cite{FranklTokushige}}] \label{thm:cross-intersecting}
Let $\cF \subseteq \binom{[n]}{a}$ and $\cG \subseteq \binom{[n]}{b}$ be cross-intersecting families: every set in $\cF$ intersects every set in $\cG$.
If $n \geq a+b$ and $b \geq a$ then
\[
|\cF| + |\cG| \leq \binom{n}{b} - \binom{n-a}{b} + 1 \leq \binom{n}{b}.
\]
\end{theorem}

We can now prove our stability result.

\begin{proof}[Proof of Theorem~\ref{thm:tEKR-stability}]
In what follows, all big~$O$ notations depend on $t$ and $\lambda$. We can assume that $n$ is large enough (as a function of $t$ and $\lambda$), since otherwise the theorem is trivial. We use the parameter $p = (k-t+1)/n$ which satisfies $\lambda/2 < p < 1/(t+1)$.

Let $f$ be the characteristic function of $\cF$, so that $\EE[f] = m - \epsilon$, where $m = \binom{n-t}{k-t}/\binom{n}{k}$. Lemma~\ref{lem:stability-fourier} shows that $\|f^{>t}\|^2 = O(\max(k/r,1)) \epsilon$, and so Theorem~\ref{thm:kindler-safra} shows that $\|f-g\|^2 \leq \delta$ for the characteristic function $g$ of some family $\cG$ depending on $J = J_t$ coordinates, for some constant $J_t$, where $\delta = O(\max((k/r)^{1/C},1)\epsilon^{1/C} + 1/n^{1/C})$; here we use the fact that $\lambda \leq k/n \leq 1/2$. We want to show that if $\delta$ is small enough (as a function of $t$) then $\cG$ must be a $t$-star or a $(t,1)$-Frankl family; if $\delta$ is large then the theorem becomes trivial.

We start by showing that if $\delta$ is small enough then $\cG$ must be $t$-intersecting. Suppose without loss of generality that $\cG$ depends only on the first $J$ coordinates. We will show that $\cJ = \cG|_{[J]} \subseteq \{0,1\}^J$ must be $t$-intersecting. If $\cJ$ is not $t$-intersecting, then pick $A,B \in \cJ$ which are not $t$-intersecting, with $|A| \geq |B|$. Let $\cA = \{ S \in \binom{[n] \setminus [J]}{k-|A|} : A \cup S \in \cF \}$ and $\cB = \{ S \in \binom{[n] \setminus [J]}{k-|B|} : B \cup S \in \cF \}$. Since $n \geq (t+1)k - (t^2-1)$ and $k \geq \lambda n$, if $n$ is large enough then $(k-|A|) + (k-|B|) \leq n - 2J$, and so Theorem~\ref{thm:cross-intersecting} shows that $|\cA| + |\cB| \leq \binom{n-J}{k-|B|}$. Therefore
\[
 \|f-g\|^2 = \frac{|\cF \triangle \cG|}{\binom{n}{k}} \geq \frac{\binom{n-J}{k-|A|}}{\binom{n}{k}} = p^{|A|} (1-p)^{J-|A|} \left(1 \pm O\left(\frac{1}{p(1-p)n}\right)\right) = \Omega(1),
\]
using Lemma~\ref{lem:measures-close} (for large enough $n$) and the fact that $p > \lambda/2$. We conclude that if $\delta$ is small enough, $\cJ = \cG|_{[J]}$ must be $t$-intersecting. 

Next, we show that if $\delta$ is small enough then $\cG$ must be either a $t$-star or a $(t,1)$-Frankl family. If $\cG$ is neither then $\mu_p(\cJ) < p^t$ for all $0 < p \leq 1/(t+1)$ by Theorem~\ref{thm:tEKR-mup}, and in particular, since $p > \lambda/2$, $\mu_p(\cJ) \leq p^t - \gamma$ for some $\gamma > 0$; here we use the fact that there are finitely many $t$-intersecting families on $J$ points. Since $\nu_k(\cJ) = \mu_p(\cJ) (1 \pm O(1/n))$ due to Lemma~\ref{lem:measures-close}, for large enough $n$ and small enough $\epsilon$ we have
\[
 \|f-g\|^2 \geq (\EE[f]-\EE[g])^2 \geq (\gamma(1\pm O(1/n)) - \epsilon)^2 = \Omega(1).
\]
We deduce that if $n$ is large enough and $\epsilon$ is small enough then $\cG$ is either a $t$-star or a $(t,1)$-Frankl family.

It remains to show that if $\epsilon \leq A_{t,\lambda} \min(r/k,1)^{C+1}$ then $\cG$ cannot be a $(t,1)$-Frankl family.
Define $\tau = \min(r/k,1)$.
Let $m_1$ be the measure of a $(t,1)$-Frankl family. Lemma~\ref{lem:measure-deficit} shows that $m - m_1 = \Omega(\tau)$ (since $p > \lambda/2$ implies $m = \Omega(1)$). Therefore if $\cG$ is a $(t,1)$-Frankl family then $\EE[g] \leq m - \Omega(\tau)$. On the other hand, $\EE[g] \geq \EE[f] - \delta = m - \epsilon - O((\epsilon/\tau)^{1/C} + 1/n^{1/C})$. Put together, we obtain
\[ \Omega(\tau) \leq \epsilon + O((\epsilon/\tau)^{1/C} + 1/n^{1/C}). \]
Choose a constant $c$ so that $\epsilon \leq c\tau$ implies
\[ \Omega(\tau) \leq O((\epsilon/\tau)^{1/C} + 1/n^{1/C}); \]
if $\epsilon > c\tau$ then the theorem becomes trivial. The inequality implies that $\tau^C = O(\epsilon/\tau)$ and so $\tau^{C+1} = O(\epsilon)$, contradicting our assumption on $\epsilon$ for an appropriate choice of $A_{t,\lambda}$.
\end{proof}

Our conjecture on the optimal error bound in Theorem~\ref{thm:kindler-safra} implies an error bound of $O_{t,\lambda}(\max(k/r,1)\epsilon)$ in Theorem~\ref{thm:tEKR-stability}.

\section{Open problems} \label{sec:open-problems}

Our work gives rise to several open questions.

\begin{enumerate}
 \item Prove (or refute) an invariance principle comparing $\nu_{pn}$ and $\gamma_{p,p}$ for arbitrary (non-harmonic) multilinear polynomials.
 \item Prove a tight version of the Kindler--Safra theorem on the slice (Theorem~\ref{thm:kindler-safra}).
 \item The uniform distribution on the slice is an example of a negatively associated vector of random variables. Generalize the invariance principle to this setting.
 \item The slice $\binom{[n]}{k}$ can be thought of as a $2$-coloring of $[n]$ with a given histogram. Generalize the invariance principle to $c$-colorings with given histogram.
 \item The slice $\binom{[n]}{k}$ has a $q$-analog: all $k$-dimensional subspaces of $\FF_q^n$ for some prime power $q$. The analog of the Boolean cube consists of all subspaces of $\FF_q^n$ weighted according to their dimension. Generalize the invariance principle to the $q$-analog, and determine the analog of Gaussian space.
\end{enumerate}

\subsection*{Acknowledgements}
This paper started its life when all authors were members of a semester-long program on ``Real Analysis in Computer Science" at the Simons Institute for Theory of Computing at U.C.\ Berkeley. The authors would like to thank the institute for enabling this work.

Y.F. would like to mention that this material is based upon work supported by the National Science Foundation under agreement No.~DMS-1128155. Any opinions, findings and conclusions or recommendations expressed in this material are those of the authors, and do not necessarily reflect the views of the National Science Foundation. The bulk of the work on this paper was done while at the Institute for Advanced Study, Princeton, NJ.

E.M. would like to acknowledge the support of the following grants: NSF grants DMS 1106999 and CCF 1320105, DOD ONR grant N00014-14-1-0823, and grant 328025 from the Simons Foundation.

K.W. would like to acknowledge the support of NSF grant CCF 1117079.

\bibliographystyle{plain}
\bibliography{main}

\begin{thebibliography}{10}

\bibitem{AK2}
Rudolf Ahlswede and Levon~H. Khachatrian.
\newblock The complete intersection theorem for systems of finite sets.
\newblock {\em Eur. J. Comb.}, 18(2):125--136, 1997.

\bibitem{AK4}
Rudolf Ahlswede and Levon~H. Khachatrian.
\newblock The diametric theorem in {H}amming spaces---optimal anticodes.
\newblock {\em Adv. Appl. Math.}, 20(4):429--449, 1998.

\bibitem{AK3}
Rudolf Ahlswede and Levon~H. Khachatrian.
\newblock A pushing-pulling method: New proofs of intersection theorems.
\newblock {\em Combinatorica}, 19(1):1--15, 1999.

\bibitem{BannaiIto}
Eiichi Bannai and Tatsuro Ito.
\newblock {\em Algebraic Combinatorics {I}: {A}ssociation schemes}.
\newblock Benjamin/Cummings Pub. Co., 1984.

\bibitem{Borell}
C.~Borell.
\newblock Geometric bounds on the {O}rnstein--{U}hlenbeck velocity process.
\newblock {\em Z. Wahrsch. Verw. Gebiete}, 70(1):1--13, 1985.

\bibitem{DDGKS}
Roee David, Irit Dinur, Elazar Goldenberg, Guy Kindler, and Igor Shinkar.
\newblock Direct sum testing.
\newblock In {\em ITCS 2015}, 2015.

\bibitem{DMN}
Anindya De, Elchanan Mossel, and Joe Neeman.
\newblock Majority is stablest: discrete and {S}o{S}.
\newblock In {\em 45th {ACM} Symposium on Theory of Computing}, pages 477--486,
  2013.

\bibitem{DinurSafra}
Irit Dinur and Shmuel Safra.
\newblock On the hardness of approximating minimum vertex cover.
\newblock {\em Ann. Math.}, 162(1):439--485, 2005.

\bibitem{Dunkl76}
Charles~F. Dunkl.
\newblock A {K}rawtchouk polynomial addition theorem and wreath products of
  symmetric groups.
\newblock {\em Indiana Univ. Math. J.}, 25:335--358, 1976.

\bibitem{Dunkl79}
Charles~F. Dunkl.
\newblock Orthogonal functions on some permutation groups.
\newblock In {\em Relations between combinatorics and other parts of
  mathematics}, volume~34 of {\em Proc. Symp. Pure Math.}, pages 129--147,
  Providence, RI, 1979. Amer. Math. Soc.

\bibitem{EKR}
Paul Erd\H{o}s, Chao Ko, and Richard Rado.
\newblock Intersection theorems for systems of finite sets.
\newblock {\em Quart. J. Math. Oxford}, 12(2):313--320, 1961.

\bibitem{FThesis}
Yuval Filmus.
\newblock {\em Spectral methods in extremal combinatorics}.
\newblock PhD thesis, University of Toronto, 2013.

\bibitem{F}
Yuval Filmus.
\newblock An orthogonal basis for functions over a slice of the boolean
  hypercube.
\newblock {\em Elec. J. Comb.}, 23(1):P1.23, 2016.

\bibitem{Frankl87}
P\'eter Frankl.
\newblock {E}rd{\H{o}}s-{K}o-{R}ado theorem with conditions on the maximal
  degree.
\newblock {\em J. Comb. Theory A}, 46:252--263, 1987.

\bibitem{FLST}
Peter Frankl, Sang~June Lee, Mark Siggers, and Norihide Tokushige.
\newblock An {E}rd{\H{o}}s--{K}o--{R}ado theorem for cross-$t$-intersecting
  families.
\newblock {\em J. Combin. Th., Ser. A}, 128:207--249, 2014.

\bibitem{FranklTokushige}
Peter Frankl and Norihide Tokushige.
\newblock Some best-possible inequalities concerning cross-intersecting
  families.
\newblock {\em J. Combin. Th., Ser. A}, 61:87--97, 1992.

\bibitem{Friedgut}
Ehud Friedgut.
\newblock On the measure of intersecting families, uniqueness and stability.
\newblock {\em Combinatorica}, 28(5):503--528, 2008.

\bibitem{IM}
Marcus Isaksson and Elchanan Mossel.
\newblock Maximally stable {G}aussian partitions with discrete applications.
\newblock {\em Israel J. Math.}, 189(1):347--396, 2012.

\bibitem{KhotSurvey}
Subhash Khot.
\newblock Inapproximability of {NP}-complete problems, discrete fourier
  analysis, and geometry.
\newblock In {\em Proceedings of the International Congress of Mathematicians},
  Hyderabad, India, 2010.

\bibitem{Kindler}
Guy Kindler.
\newblock {\em Property testing, {PCP} and Juntas}.
\newblock PhD thesis, Tel-Aviv University, 2002.

\bibitem{KKO}
Guy Kindler, Naomi Kirshner, and Ryan O'Donnell.
\newblock Gaussian noise sensitivity and {F}ourier tails, 2014.
\newblock Manuscript.

\bibitem{KindlerSafra}
Guy Kindler and Shmuel Safra.
\newblock Noise-resistant {B}oolean functions are juntas, 2004.
\newblock Unpublished manuscript.

\bibitem{LeeYau}
Tzong-Yau Lee and Horng-Tzer Yau.
\newblock Logarithmic {S}obolev inequality for some models of random walks.
\newblock {\em Ann. Prob.}, 26(4):1855--1873, 1998.

\bibitem{Lovasz}
L\'aszl\'o Lov\'asz.
\newblock On the {S}hannon capacity of a graph.
\newblock {\em IEEE Trans. Inform. Theory}, 25:1--7, 1979.

\bibitem{FilmusMossel}
Elchanan Mossel and Yuval Filmus.
\newblock Harmonicity and invariance on slices of the {B}oolean cube.
\newblock In {\em 31st Conf. Comp. Comp.}, 2016.

\bibitem{MOO}
Elchanan Mossel, Ryan O'Donnell, and Krzysztof Oleszkiewicz.
\newblock Noise stability of functions with low influences: Invariance and
  optimality.
\newblock {\em Ann. Math.}, 171:295--341, 2010.

\bibitem{NisanSzegedy}
Noam Nisan and Mario Szegedy.
\newblock On the degree of boolean functions as real polynomials.
\newblock {\em Comp. Comp.}, 4(4):301--313, 1994.

\bibitem{Ryan}
Ryan O'Donnell.
\newblock {\em Analysis of {B}oolean Functions}.
\newblock Cambridge University Press, 2014.

\bibitem{QiuZhan}
Li~Qiu and Xingzhi Zhan.
\newblock On the span of {H}adamard products of vectors.
\newblock {\em Linear Algebra Appl.}, (422):304--307, 2007.

\bibitem{Srinivasan}
Murali~K. Srinivasan.
\newblock Symmetric chains, {G}elfand--{T}setlin chains, and the {T}erwilliger
  algebra of the binary {H}amming scheme.
\newblock {\em J. Algebr. Comb.}, 34(2):301--322, 2011.

\bibitem{Tanaka}
Hajime Tanaka.
\newblock A note on the span of {H}adamard products of vectors.
\newblock {\em Linear Algebra Appl.}, (430):865--867, 2009.

\bibitem{Wilson}
Richard~M. Wilson.
\newblock The exact bound in the {E}rd{\H{o}}s-{K}o-{R}ado theorem.
\newblock {\em Combinatorica}, 4:247--257, 1984.

\bibitem{Wimmer}
Karl Wimmer.
\newblock Low influence functions over slices of the {B}oolean hypercube depend
  on few coordinates.
\newblock In {\em CCC}, 2014.

\end{thebibliography}

\end{document}